\documentclass[12pt,onecolumn, draftclsnofoot]{IEEEtran}
\usepackage{bbm}
\usepackage{mathrsfs}
\usepackage{amsfonts}
\usepackage{epsfig}
\usepackage{psfrag}
\usepackage{graphicx}
\usepackage{epstopdf}
\usepackage{cite}
\usepackage{caption}

\usepackage{array}
\usepackage{stfloats}
\usepackage{textcomp}
\usepackage{verbatim}
\usepackage[caption=false,font=normalsize,labelfont=sf,textfont=sf]{subfig}
\usepackage[justification=centering]{caption}
\usepackage{multirow}
\usepackage{makecell}

\usepackage{amssymb}
\usepackage{amsthm}
\usepackage{algorithm,algorithmic}
\usepackage[tbtags]{amsmath}
\usepackage{url}
\usepackage{bm}

\usepackage{color}
\definecolor{gray}{RGB}{128,128,128}

\allowdisplaybreaks

\newtheorem{theorem}{Theorem}
\newtheorem{assumption}{Assumption}

\newtheorem{lemma}{Lemma}

\newtheorem{remark}{Remark}

\renewcommand{\IEEEproof}{\textbf{Proof}.}

\usepackage{color}

\definecolor{gray}{RGB}{128,128,128}

\allowdisplaybreaks

\begin{document}
\title{\textbf{Distributed Constrained Online Nonconvex Optimization with Compressed Communication}}
\author{Kunpeng~Zhang,
        Lei~Xu,
        Xinlei~Yi,
        Ming~Cao,
        Karl~H.~Johansson,\\
        Tianyou~Chai,
        and Tao~Yang
}


\maketitle

\begin{abstract}
This paper considers distributed online nonconvex optimization with time-varying inequality constraints over a network of agents. For a time-varying graph, we propose a distributed online primal--dual algorithm with compressed communication to efficiently utilize communication resources. We show that the proposed algorithm establishes an $\mathcal{O}( {{T^{\max \{ {1 - {\theta _1},{\theta _1}} \}}}} )$ network regret bound and an $\mathcal{O}( {T^{1 - {\theta _1}/2}} )$ network cumulative constraint violation bound, where $T$ is the number of iterations and ${\theta _1} \in ( {0,1} )$ is a user-defined trade-off parameter. When Slater's condition holds, the network cumulative constraint violation bound is reduced to $\mathcal{O}( {T^{1 - {\theta _1}}} )$. 
These bounds are comparable to the state-of-the-art results established by existing distributed online algorithms with perfect communication for distributed online convex optimization with (time-varying) inequality constraints. Finally, a simulation example is presented to validate the theoretical results.
\end{abstract}


\section{Introduction}
Distributed online convex optimization offers a promising framework for modeling a variety of problems in dynamic, uncertain, and adversarial environments, with wide-ranging applications such as real-time routing in data networks and online advertisement placement in web search \cite{Hazan2016a}.
Various projection-based distributed online algorithms with sublinear regret have been proposed to solve the distributed online convex optimization problem, see, e.g., \cite{SundharRam2010, Yan2012, Tsianos2012, Hosseini2016, MateosNunez2014, Koppel2015, Shahrampour2017, Yuan2021, Li2022}, and recent survey paper~\cite{Li2023}. For example, for the fix communication topology, the authors of \cite{SundharRam2010} propose a projection-based distributed online subgradient descent algorithm, and establish an $\mathcal{O}( {\sqrt T } )$ regret bound for general convex local loss functions. For strongly convex local loss functions, the authors of \cite{Yan2012, Tsianos2012} establish an $\mathcal{O}\big( {\log ( T )} \big)$ regret bound. For time-varying communication topology, the authors of \cite{Hosseini2016} propose a projection-based distributed online weighted dual averaging algorithm, and establish an $\mathcal{O}( {\sqrt T } )$ regret bound for general convex local loss functions. By utilizing proportional-integral distributed feedback on the disagreement among neighboring agents, the authors of \cite{MateosNunez2014} propose a projection-based distributed online proportional-integral subgradient descent algorithm, and establish an~$\mathcal{O}\big( {\log ( T )} \big)$ regret bound for strongly convex local loss functions. 

The aforementioned distributed online algorithms rely on agents exchanging their local data with perfect communication. Consequently, these algorithms encounter significant limitations arising from communication bottlenecks. To overcome the limitations, distributed online algorithms with compressed communication are studied for the fix communication topology in the literature, see \cite{Cao2023a, Tu2022, Ge2023}, and recent survey paper \cite{Cao2023}. For example, the authors of \cite{Cao2023a} propose a decentralized online gradient descent algorithm with compressed communication by introducing an auxiliary variable to estimate the neighbors' decisions at each iteration. They establish an $\mathcal{O}(\sqrt{T})$ network regret bound for general convex local loss functions. Unlike the compression strategy employed in \cite{Cao2023a}, the authors of \cite{Tu2022, Ge2023} introduce two auxiliary variables: the first serves the same purpose as the auxiliary variable in \cite{Cao2023a}, while the second ensures that the first variable does not need to be exchanged. The compression strategy is effective when the communication topology is fixed. However, it becomes ineffective for a time-varying communication topology.

Note that inequality constraints are common in practical applications. However, performing projection operations onto such constraints can result in substantial computational and storage burdens. To deal with this challenge, the authors of~\cite{Yuan2017} consider the idea of long term constraints proposed in~\cite{Mahdavi2012}, where inequality constraints are allowed to be violated temporarily, with the requirement that they are ultimately satisfied over the long term. This violation is measured by a performance metric named constraint violation where the projection onto the non-negative orthant is performed after summing the constraint functions over time. Accordingly, they propose a distributed online primal--dual algorithm and establish an $\mathcal{O}( {{T^{1/2 + c}}} )$ regret bound and an $\mathcal{O}( {{T^{1 - c/2}}} )$ constraint violation bound for general convex local loss and constraint functions, where $c \in ( {0,1/2} )$ is a user-defined parameter. The regret bound is further reduced to $\mathcal{O}( {{T^c}} )$ for strongly convex local loss functions. 
The authors of \cite{Yuan2021b} use performance metric named cumulative constraint violation where the projection onto the non-negative orthant is performed before summing the constraint functions over time, which is proposed in \cite{Yuan2018}. Moreover, they establish an $\mathcal{O}\big( {{T^{\max \{ {c, 1 - c} \}}}} \big)$ regret bound and an $\mathcal{O}( {{T^{1 - c/2}}} )$ cumulative constraint violation bound with $c \in ( {0,1} )$ for quadratic local loss functions and linear constraint functions. However, \cite{Yuan2017, Yuan2021b} only consider static inequality constraints. The authors of~\cite{Yi2023} extend distributed online convex optimization with long-term constraints into the time-varying constraints setting. Moreover, the same network regret and cumulative constraint violation bounds as in \cite{Yuan2021b} are established. However, the distributed online algorithms proposed in \cite{Yuan2021b, Yi2023} are unable to achieve reduced network cumulative constraint violation under Slater's condition. 
Slater's condition is a sufficient condition for strong duality to hold in convex optimization problems \cite{Boyd2004}, and can be leveraged to achieve reduced constraint violation, e.g., \cite{Yu2017, Neely2017}. 
Recently, the authors of \cite{Yi2024} propose a novel distributed online primal--dual algorithm, and establish reduced network cumulative constraint violation bounds under Slater's condition. 

Unlike the aforementioned studies that focus on distributed online convex optimization, the authors of \cite{Lu2021a} investigate distributed online nonconvex optimization where local loss functions are nonconvex. To evaluate algorithm performance, they propose a novel regret metric based on the first-order optimality condition associated with the variational inequality. For this metric, the offline benchmark seeks a stationary point of the cumulative global loss functions across all iterations. Moreover, they establish an $\mathcal{O}(\sqrt{T})$ regret bound for general nonconvex local loss functions. However, \cite{Lu2021a} does not account for inequality constraints and uses perfect communication among agents.

Motivated by the above observations, this paper considers the distributed online nonconvex optimization problem with time-varying constraints. For a time-varying communication topology, we propose a distributed online primal--dual algorithm with compressed communication to efficiently utilize communication resources. Furthermore, base on several classes of appropriately chosen parameter sequences, we analyze how compressed communication influences network regret and cumulative constraint violation. The contributions are as follows.
\begin{itemize}
\item[$\bullet$]
To the best of our knowledge, this paper is among the first to consider (time-varying) inequality constraints for distributed online nonconvex optimization. Compared to \cite{SundharRam2010, Yan2012, Tsianos2012, Hosseini2016, MateosNunez2014, Koppel2015, Shahrampour2017, Yuan2021, Li2022, Cao2023a, Tu2022, Ge2023, Yuan2017, Yuan2021b, Yi2023, Yi2024} which focus on distributed online convex optimization, we consider distributed online nonconvex optimization where the absence of the convexity assumption on local loss functions makes the analysis more challenging. Compared to \cite{Lu2021a} which investigates distributed online nonconvex optimization, we additionally consider time-varying inequality constraints, which complicate both algorithm design and performance analysis. Moreover, similar to \cite{Cao2023a, Tu2022, Ge2023}, we use compressed communication instead of perfect communication used in \cite{Lu2021a}. Different from \cite{Cao2023a, Tu2022, Ge2023} which consider a fixed communication topology, we consider a time-varying communication topology.
\item[$\bullet$]
For the scaling parameter sequence~$\{s_t\}$ produced by $\{1/{t^{\theta _2}}\}$ with ${\theta _2} > {\theta _1}$ and ${\theta _1} \in ( {0,1} )$, we show in Theorem~1 that the proposed algorithm establishes an $\mathcal{O}( {{T^{\max \{ {1 - {\theta _1},1 + {\theta _1} - {\theta _2}} \}}}} )$ network regret bound under ${\theta _1} < {\theta _2} < 1$ and an $\mathcal{O}( {{T^{\max \{ {1 - {\theta _1},{\theta _1}} \}}}} )$ network regret bound under ${\theta _2} \ge 1$, and establishes an $\mathcal{O}( {T^{1 - {\theta _1}/2}} )$ network cumulative constraint violation bound. When ${\theta _2} \ge 1$, these bounds are the same as the results established in \cite{Yi2023, Yi2024} where the local loss functions are convex and perfect communication is used. When Slater's condition holds, we further show in Theorem~1 that the proposed algorithm establishes an reduced $\mathcal{O}( {T^{1 - {\theta _1}}} )$ network cumulative constraint violation bound. This bound is the same as the results established in \cite{Yi2024}.

\item[$\bullet$]
For the scaling parameter sequence~$\{s_t\}$ produced by $\{\mu ^t\}$ with $\mu  \in ( {0,1} )$, we show in Theorem~2 that the proposed algorithm establishes an $\mathcal{O}( \sqrt T )$ network regret bound and an $\mathcal{O}( {T^{3/4}} )$ network cumulative constraint violation bound. These bounds are the same as the results established in Theorem~1 when ${\theta _1} = 1/2$ and ${\theta _2} \ge 1$. Moreover, the network regret bound is the same as the results established in \cite{Lu2021a} where compressed communication and inequality constraints are not considered, as well as the results established in \cite{Cao2023a, Tu2022} where inequality constraints and nonconvex local loss functions are not considered. 
When Slater's condition holds, we further show that in Theorem~2 that the proposed algorithm establishes an reduced $\mathcal{O}( \sqrt T )$ network cumulative constraint violation bound. This bound is the same as the results established in Theorem~1 when ${\theta _1} = 1/2$.
\end{itemize}

The remainder of this paper is organised as follows.
Section~II presents the problem formulation.
Section~III proposes the distributed online primal--dual algorithm with compressed communication, and analyze its network regret and cumulative constraint violation bounds without and with Slater's condition, respectively.
Section~IV provides a simulation example to verify the theoretical results.
Finally, Section~V concludes this paper.
All proofs can be found in Appendix.

\textbf{Notations:} All inequalities and equalities throughout this paper are understood componentwise. 
${\mathbb{N}_ + }$, $\mathbb{R}$, ${\mathbb{R}^p}$ and $\mathbb{R}_ + ^p$ denote the sets of all positive integers, real numbers, $p$-dimensional and nonnegative vectors, respectively. 
Let $\mathbb{X}$ denote a subset of ${\mathbb{R}^p}$.
Given $m$ and $n \in {\mathbb{N}_ + }$, $[ m ]$ denotes the set $\{ {1, \cdot  \cdot  \cdot ,m} \}$, and $[m, n]$ denotes the set $\{ {m, \cdot  \cdot  \cdot ,n} \}$ for $m < n$. Given vectors $x$ and $y$, ${x^T}$ denotes the transpose of the vector $x$, and $\langle {x,y} \rangle $ denotes the standard inner. ${\mathbf{0}_p}$ and ${\mathbf{1}_p}$ denote the $p$-dimensional column vector whose components are all $0$ and $1$, respectively. $\mathrm{col}( {q_1}, \cdot  \cdot  \cdot ,{q_n} )$ denotes the concatenated column vector of ${q_i} \in {\mathbb{R}^{{m_i}}}$ for $i \in [ n ]$. ${\mathbb{B}^p}$ and ${\mathbb{S}^p}$ denote the unit ball and sphere centered around the origin in ${\mathbb{R}^p}$ under Euclidean norm, respectively. For a set $\mathbb{K} \in {\mathbb{R}^p}$ and a vector $ x \in {\mathbb{R}^p}$, ${\mathcal{P}_{\mathbb{K}}}(  x  )$ denotes the projection of the vector $x$ onto the set $\mathbb{K}$, i.e., ${\mathcal{P}_{\mathbb{K}}}( x ) = \arg {\min _{y \in {\mathbb{K}}}}{\| {x - y} \|^2}$, and $[  x  ]_+$ denotes ${\mathcal{P}_{\mathbb{R}_ + ^p}}( x )$. For a function $f$ and a vector $ x $, $\nabla f( x )$ denotes the subgradient of $f$ at $x$.

\section{Problem Formulation}
Consider the distributed online nonconvex optimization problem with time-varying constraints.
At iteration $t$, a network of $n$ agents is modeled by a time-varying directed graph ${\mathcal{G}_t} = ( {\mathcal{V},{\mathcal{E}_t}} )$ with the agent set $\mathcal{V} = [ n ]$ and the edge set ${\mathcal{E}_t} \subseteq \mathcal{V} \times \mathcal{V}$. $( {j,i} ) \in {\mathcal{E}_t}$ indicates that agent $i$ can receive information from agent $j$.
The sets of in- and out-neighbors of agent~$i$ are $\mathcal{N}_i^{\text{in}}( {{\mathcal{G}_t}} ) = \{ {j \in [ n ]|( {j,i} ) \in {\mathcal{E}_t}} \}$ and $\mathcal{N}_i^{\text{out}}( {{\mathcal{G}_t}} ) = \{ {j \in [ n ]|( {i,j} ) \in {\mathcal{E}_t}} \}$, respectively.
Let $\{ {{f_{i,t}}:\mathbb{X} \to \mathbb{R}} \}$ and $\{ {{g_{i,t}}:\mathbb{X} \to {\mathbb{R}^{{m_i}}}} \}$ be the sequences of nonconvex local loss and convex local constraint functions, respectively, where ${m_i}$ is a positive integer and ${g_{i,t}} \le {\mathbf{0}_{{m_i}}}$ is the local constraint.
Each agent $i$ selects a local decisions $\{ {{x_{i,t}} \in \mathbb{X}} \}$ without prior access to $\{ {{f_{i,t}}} \}$ and $\{ {{g_{i,t}}} \}$. 
Upon selection, the nonconvex local loss function $\{ {{f_{i,t}}} \}$ and convex local constraint function $\{ {{g_{i,t}}} \}$ are privately revealed to the agent.
The goal of the agent is to choose the decision sequence $\{ {{x_{i,t}}} \}$ for $i \in [n]$ and $t \in [T]$ such that both network regret
\begin{flalign}
{\rm{Net}\mbox{-}\rm{Reg}}( T ) &:= \frac{1}{n}\sum\limits_{i = 1}^n \Big( \sum\limits_{t = 1}^T {\langle {\nabla {f_t}( {{x_{i,t}}} ),{x_{i,t}}} \rangle } 
- \mathop {\inf }\limits_{x \in \mathcal{X}_T} \Big\langle {\sum\limits_{t = 1}^T {\nabla {f_t}( {{x_{i,t}}} )} ,x} \Big\rangle \Big), \label{regret-eq1}
\end{flalign}
and network cumulative constraint violation
\begin{flalign}
{\rm{Net}\mbox{-}\rm{CCV}}( T ) &:= \frac{1}{n}\sum\limits_{i = 1}^n {\sum\limits_{t = 1}^T {\| {{{[ {{g_t}( {{x_{i,t}}} )} ]}_ + }} \|} }, \label{CCV-eq2}
\end{flalign}
increase sublinearly, where ${f_t}( x ) = \frac{1}{n}\sum\nolimits_{j = 1}^n {{f_{j,t}}( x )} $ is the global loss function of the network at iteration~$t$, ${\mathcal{X}_T} = \{ { x :x \in \mathbb{X}, {g_t}( x ) \le {\mathbf{0}_m},\forall t \in [ T ]} \}$ is the feasible set, and ${g_t}( x ) = {\rm{col}}\big( {{g_{1,t}}( x ), \cdot  \cdot  \cdot ,{g_{n,t}}( x )} \big) \in {\mathbb{R}^m}$ with $m = \sum\nolimits_{i = 1}^n {{m_i}} $ is the global constraint function of the network at iteration~$t$. Similar to existing literature on distributed online convex optimization with time-varying constraints, e.g., \cite{Yi2023, Yi2024}, we assume that the feasible set ${\mathcal{X}_T}$ is nonempty for all $T \in {\mathbb{N}_ + }$, ensuring the existence of the offline optimal static decision.

\cite{Lu2021a} proposes an individual regret metric ${\rm{Net}\mbox{-}\rm{Reg}}( T ) = \mathop {\max }\limits_{x \in \mathbb{X}} \Big( {\sum\nolimits_{t = 1}^T {\langle {\nabla {f_{i,t}}( {{x_{i,t}}} ),{x_{i,t}} - x} \rangle } } \Big)$ for distributed online nonconvex optimization by utilizing the first-order optimality condition associated with the variational inequality. In this paper, we consider time-varying inequality constraints, which cause the feasible set to become ${\mathcal{X}_T}$ instead of $\mathbb{X}$. Furthermore, the objective is to optimize the network-wide accumulated loss over all iterations, rather than the local one as considered in \cite{Lu2021a}. 
Therefore, we made a slight modification to the form of the regret metric in \cite{Lu2021a}, transforming it into the form presented in \eqref{regret-eq1}.


The following commonly used assumptions are made throughout this paper.
\begin{assumption}
The set $\mathbb{X}$ is convex and closed.
Moreover, it is bounded by $R( \mathbb{X} )$, i.e., for any $x \in \mathbb{X}$
\begin{flalign}
\| x \| \le R( \mathbb{X}). \label{ass1-eq1}
\end{flalign}
\end{assumption}


\begin{assumption}
For all $i \in [n]$, $t \in {\mathbb{N}_ + }$, the subgradients $\nabla {f_{i,t}}( x )$ and $\nabla {g_{i,t}}( x )$ exist. Moreover, there exist constants ${G_1}$ and ${G_2}$ such that
\begin{subequations}
\begin{flalign}
\| {\nabla {f_{i,t}}( x )} \| &\le {G_1}, \label{ass4-eq1a}\\
\| {\nabla {g_{i,t}}( x )} \| &\le {G_2}, x \in \mathbb{X}. \label{ass4-eq1b}
\end{flalign}
\end{subequations}
\end{assumption}

Due to the convexity of the local constraint function ${{g_{i,t}}}$, from Assumption~2 and Lemma~2.6 in \cite{ShalevShwartz2012}, for all $i \in [n]$, $t \in {\mathbb{N}_ + }$, we have
\begin{flalign}
\| {{g_{i,t}}( x ) - {g_{i,t}}( y )} \| &\le {G_2}\| {x - y} \|, x,y \in \mathbb{X}. \label{ass4-eq2}
\end{flalign}

\begin{assumption}
For all $i \in [n]$, $t \in {\mathbb{N}_ + }$, there exists a constant ${L}$ such that
\begin{flalign}
\| {\nabla {f_{i,t}}( x ) - \nabla {f_{i,t}}( y )} \| \le {L}\| {x - y} \|, x,y \in \mathbb{X}. \label{ass5-eq1}
\end{flalign}
\end{assumption}

\begin{assumption}
For $t \in {\mathbb{N}_ + }$, the time-varying directed graph $\mathcal{G}_t$ satisfies that

\noindent (i) There exists a constant $w  \in ( {0,1} )$ such that ${[ {{W_t}} ]_{ij}} \ge w$ if $( {j,i} ) \in {\mathcal{E}_t}$ or $i = j$, and ${[ {{W_t}} ]_{ij}} = 0$ otherwise.

\noindent (ii) The mixing matrix ${W_t}$ is doubly stochastic, i.e., ${\sum\nolimits_{i = 1}^n {[ {{W_t}} ]} _{ij}} = {\sum\nolimits_{j = 1}^n {[ {{W_t}} ]} _{ij}} = 1$, $\forall i,j \in [ n ]$.

\noindent (iii) There exists an integer $B > 0$ such that the time-varying directed graph $( {\mathcal{V}, \cup _{l = 0}^{B - 1}{\mathcal{E} _{t + l}}} )$ is strongly connected.
\end{assumption}

\begin{assumption}
The compressor $\mathcal{C}:{\mathbb{R}^p} \to {\mathbb{R}^p}$ satisfies
\begin{flalign}
{\mathbf{E}_\mathcal{C}}[ {\| {\mathcal{C}( x ) - x} \|_d^2} ] \le C, \forall x \in {\mathbb{R}^p}, \label{compre-eq1}
\end{flalign}
for some real norm parameter $d \ge 1$ and constant $C \ge 0$. Here ${\mathbf{E}_\mathcal{C}}$ denotes the expectation over the internal randomness of the stochastic compression operator $\mathcal{C}$.
\end{assumption}

\begin{assumption}
There exists a point ${x_s} \in \mathbb{X}$ and a positive constant ${\varsigma _s}$ such that
\begin{flalign}
{g_t}( {{x_s}} ) \le  - {\varsigma _s}{\mathbf{1}_m},t \in {\mathbb{N}_ + }. \label{ass8-eq1}
\end{flalign}
\end{assumption}

\section{Distributed Online Primal--Dual Algorithm with \\Compressed Communication}
\subsection{Algorithm Description}
To achieve reduced network cumulative constraint violation, \cite{Yi2024} proposes a distributed online primal--dual algorithm for distributed online convex optimization with time-varying constraints. Here, we first give a subgradient descent variant of this algorithm in the following: 
\begin{subequations}
\begin{flalign}
 {x_{i,t}} &= \sum\limits_{j = 1}^n {{{[{W_t}]}_{ij}}{z_{j,t}}}, \label{Algorithm0-eq1} \\
 {v_{i,t + 1}} &= {\gamma _t}{[ {{g_{i,t}}( {{x_{i,t}}} )} ]_ + }, \label{Algorithm0-eq2} \\
 {z_{i,t + 1}} &= {\mathcal{P}_{\mathbb{X}}}( {{x_{i,t}} - {\alpha _t}{{ \omega }_{i,t + 1}}} ), \label{Algorithm0-eq3} \\
 {{\omega} _{i,t + 1}} &= {\nabla} {f_{i,t}}({x_{i,t}}) + \big({\nabla} {{g_{i,t}}({x_{i,t}}) }\big)^{T}{v_{i,t+1}}, \label{Algorithm0-eq4} 
\end{flalign} \label{Algorithm0} 
\end{subequations}
where ${{\gamma _t}}$ is the regularization parameter; and ${{\alpha _t}}$ is the stepsize.
To improve communication efficiency, the distributed online primal--dual algorithm with compressed communication is proposed by using the class of compressors satisfying Assumption~5, which is presented in pseudo-code as Algorithm~1. 
\begin{algorithm}[!t]
  \caption{Distributed Online Primal--Dual Algorithm with Compressed Communication} 
  \begin{algorithmic}
  \renewcommand{\algorithmicrequire}{\textbf{Input:}}
  \REQUIRE
   non-increasing stepsize sequence $\{ {\alpha _t}\} \subseteq ( {0, + \infty })$, non-increasing scaling parameter sequence $\{ {{s _t}} \} \subseteq ( {0, + \infty })$, and non-decreasing regularization parameter sequence $\{ {\gamma _t}\} \subseteq ( {0, + \infty })$.
  \renewcommand{\algorithmicrequire}{\textbf{Initialize:}}
  \REQUIRE
      ${\hat{z}_{j,0}} = {\mathbf{0}_p}$ for $j \in [ n ]$, and ${z_{i,1}} \in {\mathbb{X}}$, $\forall i \in [ n ]$.
    \FOR {$t = 1, \cdot  \cdot  \cdot $}
    \FOR {$i = 1,\cdot  \cdot  \cdot,n$ in parallel}
    \STATE Broadcast $\mathcal{C}\big( {( {{z_{i,t}} - {\hat{z}_{i,t-1}}} )/{s_{t}}} \big)$ to $\mathcal{N}_i^{\text{out}}( {{\mathcal{G}_t}} )$ and receive $\mathcal{C}\big( {( {{z_{j,t}} - {\hat{z}_{j,t-1}}} )/{s_{t}}} \big)$ from $j \in \mathcal{N}_i^{\text{in}}( {{\mathcal{G}_t}} )$.
    \STATE Update
     \begin{flalign}
      {{\hat z}_{j,t}} = {\mathcal{P}_\mathbb{X}}\big( {{{\hat z}_{j,t - 1}} + {s_t}{\cal C}(({z_{j,t}} - {{\hat z}_{j,t - 1}})/{s_t})} \big), 
         j \in \big\{\mathcal{N}_i^{\text{in}}( {{\mathcal{G}_t}} ) \cup \{i\}\big\}. \label{Algorithm1-eq1}
    \end{flalign}
    \STATE Select
     \begin{flalign}
       {x_{i,t}} = \sum\limits_{j = 1}^n {{{[ {{W_t}} ]}_{ij}}{{\hat z}_{j,t}}}. \label{Algorithm1-eq3}
    \end{flalign}
    \STATE Observe $\nabla {f_{i,t}}( {{x_{i,t}}} )$, $\nabla {g_{i,t}}( {{x_{i,t}}} )$, and ${g_{i,t}}( {{x_{i,t}}} )$.
    \STATE Update
    \begin{subequations}
     \begin{flalign}
       {v_{i,t + 1}} &= {\gamma _t}{[ {{g_{i,t}}( {{x_{i,t}}} )} ]_ + }, \label{Algorithm1-eq4} \\
       {{\omega} _{i,t + 1}} &= \nabla {f_{i,t}}({x_{i,t}}) + \big({\nabla} {{g_{i,t}}({x_{i,t}}) }\big)^{T}{v_{i,t+1}}, \label{Algorithm1-eq5} \\
       {z_{i,t + 1}} &= {\mathcal{P}_{\mathbb{X}}}( {{x_{i,t}} - {\alpha _t}{{\omega }_{i,t + 1}}} ). \label{Algorithm1-eq6} 
      \end{flalign}
     \end{subequations}
    \ENDFOR
    \ENDFOR
  \renewcommand{\algorithmicensure}{\textbf{Output:}}
  \ENSURE
      $\{ x_{i,t} \}$.
  \end{algorithmic}
\end{algorithm}

\subsection{Performance Analysis}
In this section, we establish network regret and cumulative constraint violation bounds for Algorithm~1 in the following theorems without and with Slater's condition, respectively. Firstly, we choose the scaling parameter sequence~$\{s_t\}$ produced by $\{1/{t^{\theta _2}}\}$ in the following theorem. 

\begin{theorem}\label{thm1}
Suppose Assumptions 1--5 hold. For all $i \in [ n ]$, let $\{ {{x_{i,t}}} \}$ be the sequences generated by Algorithm~1 with
\begin{flalign}
{\alpha _t} = \frac{{\alpha _0}}{{{t^{{\theta _1}}}}}, {\gamma _t} = \frac{{{\gamma _0}}}{{{\alpha _t}}}, {s_t} = \frac{{s_0}}{{{t^{{\theta _2}}}}}, \label{theorem1-eq1}
\end{flalign}
where ${\theta _1} \in ( {0,1} )$, ${\alpha _0} > 0$, ${\gamma _0} \in ( {0,1/( {4G_2^2} )} ]$, ${s_0} > 0$, and ${\theta _2} > {\theta _1}$ are constants. Then, for any $T \in {\mathbb{N}_ + }$,
\begin{flalign}
&\mathbf{E}_\mathcal{C}[{{\rm{Net}\mbox{-}\rm{Reg}}( T )}] = \left\{ \begin{array}{l}
\mathcal{O}( {{T^{\max \{ {1 - {\theta _1},1 + {\theta _1} - {\theta _2}} \}}}} ), {\rm{if}}\;{\theta _1} < {\theta _2} < 1, \\
\mathcal{O}( {{T^{\max \{ {1 - {\theta _1},{\theta _1}} \}}}} ), {\rm{if}}\;{\theta _2} \ge 1,
\end{array} \right. \label{theorem1-eq2}\\
&\mathbf{E}_\mathcal{C}[{{\rm{Net} \mbox{-} \rm{CCV}}( T )}] = \mathcal{O}( {T^{1 - {\theta _1}/2}} ). \label{theorem1-eq3}
\end{flalign}
Moreover, if Assumption~6 also holds, then
\begin{flalign}
\mathbf{E}_\mathcal{C}[{{\rm{Net} \mbox{-} \rm{CCV}}( T )}] &= \mathcal{O}( {T^{1 - {\theta _1}}} ). \label{theorem1-eq4}
\end{flalign}
\end{theorem}
\begin{remark}\label{rem2}
We show in Theorem~1 that Algorithm~1 establishes sublinear network regret and cumulative constraint violation bounds as in \eqref{theorem1-eq2}--\eqref{theorem1-eq3}. These bounds characterize the impact of compressed communication on the network regret and cumulative constraint violation bounds, which is captured by ${\theta _2}$. When ${\theta _2} \ge 1$, they are the same as the state-of-the-art results established by the distributed online algorithms without compressed communication in \cite{Yi2023, Yi2024}. In addition, when Slater's condition holds, the network cumulative constraint violation bound is further reduced as in \eqref{theorem1-eq4}. The bound remains the same as the results established by the distributed online algorithm with perfect communication in \cite{Yi2024}.
\end{remark}

We then choose the scaling parameter sequence~$\{s_t\}$ produced by $\{\mu ^t\}$, which is also adopted by the distributed algorithms in \cite{Yi2022, Ge2023}.
\begin{theorem}\label{thm2}
Suppose Assumptions 1--5 hold. For all $i \in [ n ]$, let $\{ {{x_{i,t}}} \}$ be the sequences generated by Algorithm~1 with
\begin{flalign}
{\alpha _t} = {\alpha _0}\sqrt {\frac{{{\Psi _t}}}{t}}, {\gamma _t} = \frac{{{\gamma _0}}}{{{\alpha _t}}}, {s_t} = {s_0}{\mu ^t}, \label{theorem2-eq1}
\end{flalign}
where ${\Psi _t} = \sum\nolimits_{k = 1}^t {{\mu ^k}}$, ${\alpha _0} > 0$, ${\gamma _0} \in ( {0,1/( {4G_2^2} )} ]$, ${s_0} > 0$, and $\mu  \in ( {0,1} )$ are constants. Then, for any $T \in {\mathbb{N}_ + }$,
\begin{flalign}
&\mathbf{E}_\mathcal{C}[{{\rm{Net}\mbox{-}\rm{Reg}}( T )}] = \mathcal{O}( \sqrt T ), \label{theorem2-eq2}\\
&\mathbf{E}_\mathcal{C}[{{\rm{Net} \mbox{-} \rm{CCV}}( T )}] = \mathcal{O}( {T^{3/4}} ). \label{theorem2-eq3}
\end{flalign}
Moreover, if Assumption~6 also holds, then
\begin{flalign}
\mathbf{E}_\mathcal{C}[{{\rm{Net} \mbox{-} \rm{CCV}}( T )}] &= \mathcal{O}( \sqrt T ). \label{theorem2-eq4}
\end{flalign}
\end{theorem}
\begin{remark}\label{rem3}
We show in Theorem~2 that Algorithm~1 establishes an $\mathcal{O}( \sqrt T )$ network regret bound as in \eqref{theorem2-eq2} and an $\mathcal{O}( {T^{3/4}} )$ cumulative constraint violation bound as in \eqref{theorem2-eq3}. These bounds are the same as the results established in \eqref{theorem1-eq2}--\eqref{theorem1-eq3} with ${\theta _1} = 1/2$ and ${\theta _2} \ge 1$. Moreover, the network regret bound is the same as the results established in \cite{Lu2021a} where compressed communication and inequality constraints are not considered, and the results established in \cite{Cao2023a, Tu2022} where inequality constraints and nonconvex local loss functions are not considered. In addition, when Slater's condition holds, the network cumulative constraint violation bound is further reduced as in \eqref{theorem2-eq4}, which is the same as the results established in \eqref{theorem1-eq4} with ${\theta _1} = 1/2$.
\end{remark}

\section{Simulation Example}
To evaluate the performance of Algorithm~1, we consider a distributed
online localization problem with long-term constraints over a network of $100$ sensors as follows: 
\begin{subequations}
\begin{flalign}
&\mathop {\min }\limits_x \;\;\;\sum\limits_{t = 1}^T {\sum\limits_{i = 1}^n {\frac{1}{4}} } {\big| {{{\| {{S_i} - x} \|}^2} - {D_{i,t}}} \big|^2}, \label{example1-1}\\ 
&\;{\rm{s}}{\rm{.t}}{\rm{.   }}\;\;\;\;x \in \mathbb{X},{\rm{ }}{B_{i,t}}x - {b_{i,t}} \le {\mathbf{0}_{{m_i}}},\forall i \in [ n ],\forall t \in [ T ], \label{example1-2}
\end{flalign} \label{example1}%
\end{subequations}
where $S_i \in {\mathbb{R}^p}$ denotes the position of sensor~$i$, ${D_{i,t}} = {\| {{S_i} - {X_{0,t}}} \|^2} + {\tau _{i,t}}$ denotes the distance between the positions of the target and sensor $i$, ${X_{0,t}} \in {\mathbb{R}^p}$ and ${\tau _{i,t}} \in \mathbb{R}$ denote the position of the target and the measurement noise at iteration $t$, respectively, and ${B_{i,t}} \in {\mathbb{R}^{{m_i} \times p}}$ and ${b_{i,t}} \in {\mathbb{R}^{{m_i}}}$ denote the coefficient matrix and coefficient vector of the local linear constraints, respectively. 
The communication topology is modeled by a time-varying undirected graph.
Specifically, at each iteration~$t$, the graph is first randomly generated where the probability of any two sensors being connected is $\rho $. Then, to make sure that Assumption~4 is satisfied, we add edges~$( {i,i + 1} )$ for $i \in [ 24 ]$ when $t \in \{ {4c + 1} \}$, edges~$( {i,i + 1} )$ for $i \in [ 25, 49 ]$ when $t \in \{ {4c + 2} \}$, edges~$( {i,i + 1} )$ for $i \in [ 50, 74 ]$ when $t \in \{ {4c + 3} \}$, edges~$( {i,i + 1} )$ for $i \in [ 75, 99 ]$ when $t \in \{ {4c + 4} \}$ for $c = \{0, 1, \cdot  \cdot  \cdot \}$. Moreover, let
${[ {{W_t}} ]_{ij}} = \frac{1}{n}$ if $( {j,i} ) \in {\mathcal{E}_t}$ and ${[ {{W_t}} ]_{ii}} = 1 - \sum\nolimits_{j = 1}^n {{{[ {{W_t}} ]}_{ij}}} $.

In this paper, we show in Theorem~1 that, both without and with Slater's condition, Algorithm~1 establishes the same network regret and
cumulative constraint violation bounds as the state-of-the-art results on distributed online convex optimization with long-term constraints, established by the distributed online algorithms with perfect communication in~\cite{Yi2024}. To verify the theoretical results, we compare Algorithm~1 with the algorithm in \cite{Yi2024}. We set $\rho  = 0.1$, $\mathbb{X} = {[ { - 5,5} ]^p}$, $p = 2$, ${m_i} = 2$, and randomly choose each component of ${{S_i}}$ from the uniform distribution in the interval $[ { - 10,10} ]$. We assume that the position of the target evolves by
\begin{flalign}
\nonumber
{X_{0,t + 1}} = {X_{0,t}} + \left[ {\begin{array}{*{20}{c}}
{\frac{{{{( { - 1} )}^{{Q_t}}}\sin ( {t/50} )}}{{10t}}}\\
{\frac{{ - {Q_t}\cos ( {t/70} )}}{{40t}}}
\end{array}} \right],
\end{flalign}
where ${{Q_t}}$ is randomly generated from Bernoulli distribution with a success probability of $0.5$, and ${X_{0,0}} = {[ {0.8,0.95} ]^T}$. Moreover, ${\tau _{i,t}}$ is randomly generated from the uniform distribution from in the interval $[ { 0,0.001} ]$. Furthermore, each component of ${B_{i,t}}$ is randomly generated from the uniform distribution in the interval $[ {0,2} ]$, and each component of ${b_{i,t}}$ is randomly generated from the uniform distribution in the interval $[ {b,b+1} ]$ with $b > 0$. Note that $b > 0$ guarantees Slater's condition holds. Here we choose $b=0.01$.
In addition, we select the following compressor for Algorithm~1:
\begin{flalign}
\nonumber
\mathcal{C}( x ) = \Delta \Big\lfloor {\frac{x}{\Delta } + \frac{{{\mathbf{1}_p}}}{2}} \Big\rfloor,
\end{flalign}
where $\Delta $ is a positive integer. This compressor satisfies Assumption~5 with $d = \infty $ and $C = {\Delta ^2}/4$, which is also used in \cite{Yi2022, Ge2023, Li2017, Khirirat2021}. Transmitting $\mathcal{C}( x )$ requires $pq$ bits if each integer is encoded using $q$ bits. Here we set $\Delta =1$ and $q=8$.

\begin{figure}[!ht]
 \centering
  \includegraphics[width=10cm]{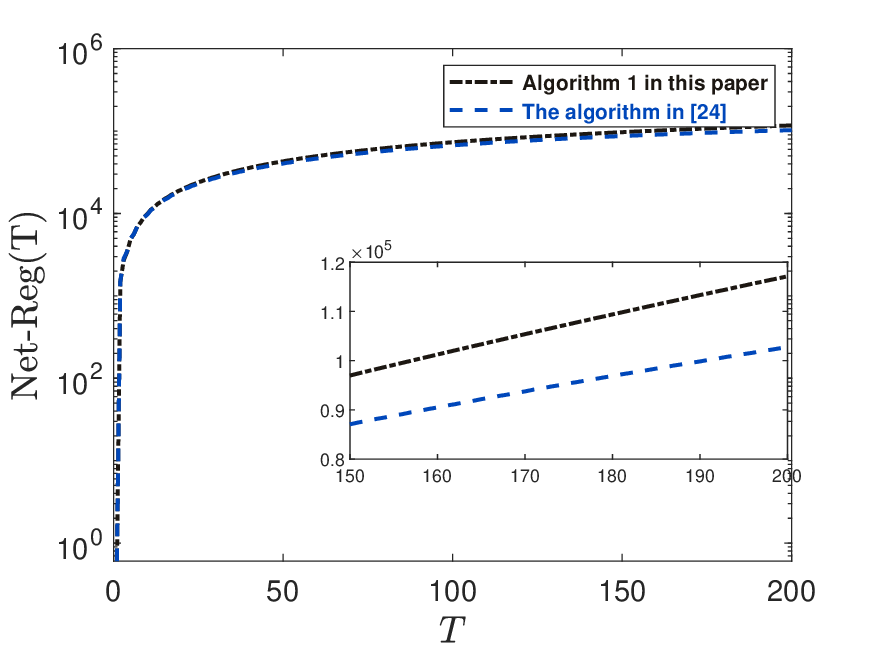}
  \caption{Evolutions of network regret.}
\end{figure}

\begin{figure}[!ht]
 \centering
  \includegraphics[width=10cm]{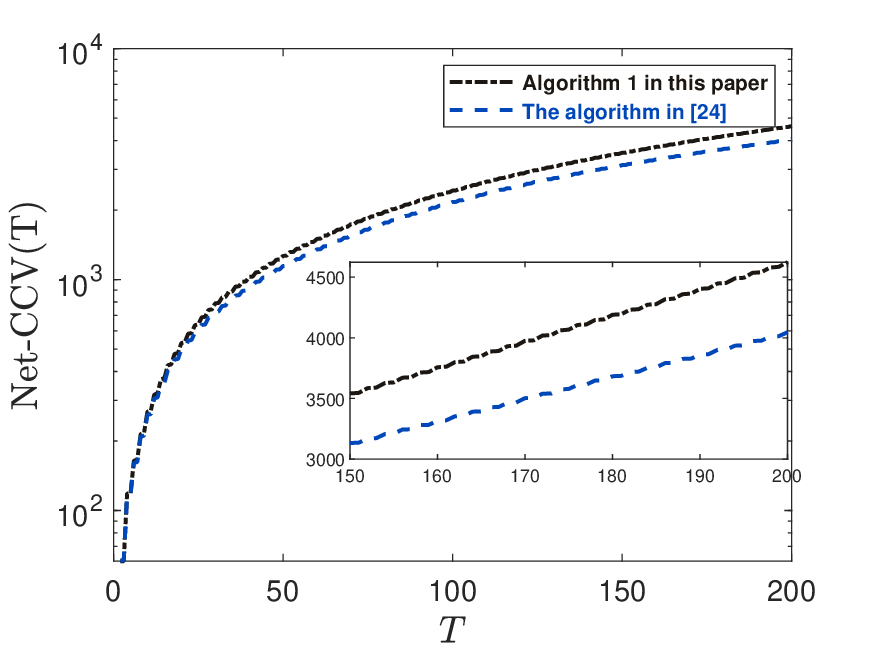}
  \caption{Evolutions of network cumulative constraint violation.}
\end{figure}
Figs.~1 and 2 illustrate the evolutions of network regret and cumulative constraint violation, respectively. As shown in Fig.1, our Algorithm~1 exhibits almost the same network regret as that of the algorithm in \cite{Yi2024}. Similarly, Fig. 2 demonstrates that our Algorithm~1 also has almost the same network cumulative constraint violation as that of the algorithm in \cite{Yi2024}. However, due to compressed communication, our Algorithm~1 requires significantly fewer bits than those required by the algorithm in \cite{Yi2024}. These simulation results are consistent with the results in Theorem~1.

\section{Conclusions}
This paper studied the distributed online nonconvex optimization problem with time-varying constraints. To better utilize communication resources, we proposed a distributed online primal--dual algorithm with compressed communication. More importantly, the algorithm was able to handle time-varying communication topologies. We showed that the algorithm established sublinear network regret and cumulative constraint violation bounds. Moreover, the network cumulative constraint violation bounds were further reduced when Slater's condition held. 
In the future, we plan to investigate distributed bandit nonconvex optimization with time-varying constraints since gradient information is unavailable in many real-world applications.

\appendix

\hspace{-3mm}\emph{A. Useful Lemmas}

We begin by presenting several preliminary results that will be utilized in the subsequent proofs.
\begin{lemma}
(Equation (5.4.21) on page 333 in \cite{Horn2012})
For any $x \in {\mathbb{R}^p}$, it holds that ${\| x \|_d} \le \hat p\| x \|$ and $\| x \| \le \tilde p{\| x \|_d}$, where $\hat p = {p^{\frac{1}{d} - \frac{1}{2}}}$ and $\tilde p = 1$ when $d \in [ {1,2} ]$, and $\hat p = 1$ and $\tilde p = {p^{\frac{1}{2} - \frac{1}{d}}}$ when $d > 2$.
\end{lemma}

\begin{lemma}
(\cite{Nedic2009, Nedic2014})
Let ${W_t}$ denote the mixing matrix associated with a time-varying graph that satisfies Assumption~4. Then,
\begin{flalign}
\Big| {{{[ {\Psi _s^t} ]}_{ij}} - \frac{1}{n}} \Big| \le \tau {\lambda ^{t - s}}, \forall i,j \in [ n ], \forall t \ge s \ge 1, \label{Lemma1-eq1}
\end{flalign}
where $\Psi _s^t = {W_t}{W_{t - 1}} \cdots {W_s}$, $\tau  = {( {1 - \omega /4{n^2}} )^{ - 2}} > 1$, and $\lambda  = {( {1 - \omega /4{n^2}} )^{1/B}} \in ( {0,1} )$.
\end{lemma}

\begin{lemma}
(\cite{Zhang2024b})
Let $\mathbb{K}$ denote a nonempty closed convex subset of ${\mathbb{R}^p}$ and let $b$ and $c$ denote two vectors in ${\mathbb{R}^p}$. If ${x} = {\mathcal{P}_\mathbb{K}}( {b - c} )$, then for all $y \in \mathbb{K}$,
\begin{flalign}
2\langle {x - y,c} \rangle  \le {\| {y - b} \|^2} - {\| {y - x} \|^2} - {\| {x - b} \|^2}. \label{Lemma2-eq1}
\end{flalign}
In addition, let $\Phi ( y ) = {\| {b - y} \|^2} + 2\langle {c,y} \rangle $, then we know $\Phi$ is a strongly convex function with convexity parameter $\sigma = 2$ and $x = \mathop {\arg \min }\limits_{y \in \mathbb{K}} \Phi ( y )$. Then, 
\begin{flalign}
\| {x - b} \| \le {\| c \|} \label{Lemma2-eq2}
\end{flalign}
holds.
\end{lemma}
\begin{lemma}
If Assumption~4 holds. For all $i \in [ n ]$ and $t \in {\mathbb{N}_ + }$, ${\hat{z}_{i,t}}$ generated by Algorithm~1 satisfy
\begin{flalign}
\| {{\hat{z}_{i,t}} - {{\bar z}_t}} \| \le \tau {\lambda ^{t - 2}}\sum\limits_{j = 1}^n {\| {{\hat{z}_{j,1}}} \|}  + \tau \sum\limits_{s = 1}^{t - 2} {{\lambda ^{t - s - 2}}} \sum\limits_{j = 1}^n {\| {\varepsilon _{j,s}^z} \|}  + \| {\varepsilon _{i,t - 1}^z} \| + \frac{1}{n}\sum\limits_{j = 1}^n {\| {\varepsilon _{j,t - 1}^z} \|}, \label{Lemma3-eq1}
\end{flalign}
where ${{\bar z}_t} = \frac{1}{n}\sum\nolimits_{i = 1}^n {{\hat{z}_{i,t}}} $, $\varepsilon _{i,t - 1}^z = {\hat{z}_{i,t}} - {x_{i,t - 1}}$.
\end{lemma}
\begin{proof}
From \eqref{Algorithm1-eq3}, we recall that
\begin{flalign}
{x_{i,t}} = \sum\limits_{j = 1}^n {{{[ {{W_t}} ]}_{ij}}{\hat{z}_{j,t}}} \label{Lemma3-proof-eq1}
\end{flalign}
holds.

From \eqref{Lemma3-proof-eq1} and $\varepsilon _{i,t - 1}^z = {\hat{z}_{i,t}} - {x_{i,t - 1}}$, we have
\begin{flalign}
{\hat{z}_{i,t}} = \sum\limits_{j = 1}^n {{{[ {{W_{t - 1}}} ]}_{ij}}{\hat{z}_{j,t - 1}}}  + \varepsilon _{i,t - 1}^z. \label{Lemma3-proof-eq2}
\end{flalign}
Then, by following the proof of Lemma~4 in \cite{Yi2023}, we conclude that \eqref{Lemma3-eq1} holds.
\end{proof}
\begin{lemma}
Suppose Assumptions~1--2 and 4--5 hold. For all $i \in [ n ]$, let $\{ {{x_{i,t}}} \}$ be the sequences generated by Algorithm~1 and $y$ be an arbitrary point in $\mathbb{X}$, then
\begin{flalign}
\nonumber
{\mathbf{E}_\mathcal{C}}[\frac{1}{n}\sum\limits_{i = 1}^n {v_{i,t + 1}^T{g_{i,t}}( {{x_{i,t}}} )}  + \frac{1}{n}\sum\limits_{i = 1}^n {\langle {\nabla {f_{i,t}}( {{x_{i,t}}} ),{x_{i,t}} - y} \rangle }]  &\le \frac{1}{n}\sum\limits_{i = 1}^n {\mathbf{E}_\mathcal{C}}[{v_{i,t + 1}^T{g_{i,t}}( y )}]  + \frac{1}{n}\sum\limits_{i = 1}^n {\mathbf{E}_\mathcal{C}}[{{\Delta _{i,t}}( y )}] \\
&\;\; + \frac{{\mathbf{E}_\mathcal{C}}[{{{\tilde \Delta }_t}}]}{n} + \frac{{2{\tilde p}\sqrt{C}R( \mathbb{X} )}{s_{t + 1}}}{{{\alpha _t}}}, \label{Lemma4-eq1}
\end{flalign}
where 
\begin{flalign}
\nonumber
&{\Delta _{i,t}}( y ) = \frac{1}{{2{\alpha _t}}}( {{{\| {y - {x_{i,t}}} \|}^2} - {{\| {y - {x_{i,t + 1}}} \|}^2}} ), \\
\nonumber
&{{\tilde \Delta }_t} = \sum\limits_{i = 1}^n {( {{G_1} + {G_2}\| {{v_{i,t + 1}}} \|} )} \| {{x_{i,t}} - {z_{i,t + 1}}} \| - \sum\limits_{i = 1}^n {\frac{{{{\| {{x_{i,t}} - {z_{i,t + 1}}} \|}^2}}}{{2{\alpha _t}}}}.
\end{flalign}
\end{lemma}
\begin{proof}
Since the local constraint function ${g_{i,t}}$ is convex, we have
\begin{flalign}
{g_{i,t}}( y ) \ge {g_{i,t}}( x ) + \nabla {g_{i,t}}( x )( {y - x} ), \forall x,y \in \mathbb{X}. \label{Lemma4-proof-eq1}
\end{flalign}
We have
\begin{flalign}
\nonumber
{\mathbf{E}_\mathcal{C}}[\langle {\nabla {f_{i,t}}( {{x_{i,t}}} ),{x_{i,t}} - y} \rangle]
& = {\mathbf{E}_\mathcal{C}}[\langle {\nabla {f_{i,t}}( {{x_{i,t}}} ),{x_{i,t}} - {z_{i,t + 1}}} \rangle]  + {\mathbf{E}_\mathcal{C}}[\langle {\nabla {f_{i,t}}( {{x_{i,t}}} ),{z_{i,t + 1}} - y} \rangle] \\
& \le {G_1}{\mathbf{E}_\mathcal{C}}[\| {{x_{i,t}} - {z_{i,t + 1}}} \|] + {\mathbf{E}_\mathcal{C}}[\langle {\nabla {f_{i,t}}( {{x_{i,t}}} ),{z_{i,t + 1}} - y} \rangle], \label{Lemma4-proof-eq2}
\end{flalign}
where the inequality holds due to \eqref{ass4-eq1a}.

For the second term on the right-hand side of \eqref{Lemma4-proof-eq2}, from \eqref{Algorithm1-eq5}, we have
\begin{flalign}
\nonumber
{\mathbf{E}_\mathcal{C}}[\langle {\nabla {f_{i,t}}( {{x_{i,t}}} ),{z_{i,t + 1}} - y} \rangle]
& = {\mathbf{E}_\mathcal{C}}\big[\big\langle {\big({\nabla} {{g_{i,t}}({x_{i,t}}) }\big)^{T}{v_{i,t + 1}},y - {z_{i,t + 1}}} \big\rangle\big]  + {\mathbf{E}_\mathcal{C}}[\langle {{\omega _{i,t + 1}},{z_{i,t + 1}} - y} \rangle] \\
\nonumber
& = {\mathbf{E}_\mathcal{C}}\big[\big\langle {\big({\nabla} {{g_{i,t}}({x_{i,t}}) }\big)^{T}{v_{i,t + 1}},y - {x_{i,t}}} \big\rangle\big]  + {\mathbf{E}_\mathcal{C}}\big[\big\langle {\big({\nabla} {{g_{i,t}}({x_{i,t}}) }\big)^{T}{v_{i,t + 1}},{x_{i,t}} - {z_{i,t + 1}}} \big\rangle\big] \\
& \;\;+ {\mathbf{E}_\mathcal{C}}[\langle {{\omega _{i,t + 1}},{z_{i,t + 1}} - y} \rangle]. \label{Lemma4-proof-eq3}
\end{flalign}
Next, we find the upper bound of each term on the right-hand side of \eqref{Lemma4-proof-eq3}.

From ${v_{i,t}} \ge {\mathbf{0}_{{m_i}}}$, $\forall i \in [ n ]$, $\forall t \in {\mathbb{N}_ + }$ and \eqref{Lemma4-proof-eq1}, we have
\begin{flalign}
{\mathbf{E}_\mathcal{C}}\big[\big\langle {\big({\nabla} {{g_{i,t}}({x_{i,t}}) }\big)^{T}{v_{i,t + 1}},y - {x_{i,t}}} \big\rangle\big]  \le {\mathbf{E}_\mathcal{C}}[v_{i,t + 1}^T{g_{i,t}}( y )] - {\mathbf{E}_\mathcal{C}}[v_{i,t + 1}^T{g_{i,t}}( {{x_{i,t}}} )]. \label{Lemma4-proof-eq4}
\end{flalign}
From the Cauchy--Schwarz inequality and \eqref{ass4-eq1b}, we have
\begin{flalign}
{\mathbf{E}_\mathcal{C}}\big[\big\langle {\big({\nabla} {{g_{i,t}}({x_{i,t}}) }\big)^{T}{v_{i,t + 1}},{x_{i,t}} - {z_{i,t + 1}}} \big\rangle\big]  \le {G_2}{\mathbf{E}_\mathcal{C}}[\| {{v_{i,t + 1}}} \|\| {{x_{i,t}} - {z_{i,t + 1}}} \|]. \label{Lemma4-proof-eq5}
\end{flalign}
By applying \eqref{Lemma2-eq1} to the update \eqref{Algorithm1-eq6}, we have
\begin{flalign}
\nonumber
&\;\;\;\;\;{\mathbf{E}_\mathcal{C}}[\langle {{\omega _{i,t + 1}},{z_{i,t + 1}} - y} \rangle] \\
\nonumber
& \le \frac{1}{{2{\alpha _t}}}( {{\mathbf{E}_c}[{{\| {y - {x_{i,t}}} \|}^2}] - {\mathbf{E}_\mathcal{C}}[{{\| {y - {z_{i,t + 1}}} \|}^2}] - {\mathbf{E}_\mathcal{C}}[{{\| {{x_{i,t}} - {z_{i,t + 1}}} \|}^2}]} ) \\
\nonumber
& = \frac{1}{{2{\alpha _t}}}( {\mathbf{E}_\mathcal{C}}[{{\| {y - {x_{i,t}}} \|}^2}] - {\mathbf{E}_\mathcal{C}}[{{\| {y - {x_{i,t + 1}}} \|}^2}] + {\mathbf{E}_\mathcal{C}}[{{\| {y - {x_{i,t + 1}}} \|}^2}] - {\mathbf{E}_\mathcal{C}}[{{\| {y - {z_{i,t + 1}}} \|}^2}] \\
\nonumber
& \;\;- {\mathbf{E}_\mathcal{C}}[{{\| {{x_{i,t}} - {z_{i,t + 1}}} \|}^2}] ) \\
\nonumber
&  = {\mathbf{E}_\mathcal{C}}[{\Delta _{i,t}}( y )] + \frac{1}{{2{\alpha _t}}}\Big( {{\mathbf{E}_\mathcal{C}}\Big[{\Big\| {y - \sum\limits_{j = 1}^n {{{[ {{W_{t + 1}}} ]}_{ij}}{\hat{z}_{j,t + 1}}} } \Big\|^2}\Big] - {\mathbf{E}_\mathcal{C}}[{{\| {y - {z_{i,t + 1}}} \|}^2}] - {\mathbf{E}_\mathcal{C}}[{{\| {{x_{i,t}} - {z_{i,t + 1}}} \|}^2}]} \Big) \\
& \le {\mathbf{E}_\mathcal{C}}[{\Delta _{i,t}}( y )] + \frac{1}{{2{\alpha _t}}}\Big( {\sum\limits_{j = 1}^n {{{[ {{W_{t + 1}}} ]}_{ij}}{\mathbf{E}_\mathcal{C}}[{{\| {y - {\hat{z}_{j,t + 1}}} \|}^2}]}  - {\mathbf{E}_\mathcal{C}}[{{\| {y - {z_{i,t + 1}}} \|}^2}] - {\mathbf{E}_\mathcal{C}}[{{\| {{x_{i,t}} - {z_{i,t + 1}}} \|}^2}}] \Big), \label{Lemma4-proof-eq6}
\end{flalign}
where the last equality holds due to \eqref{Algorithm1-eq3}; and the last inequality holds since ${W_{t + 1}}$ is doubly stochastic and ${\|  \cdot  \|^2}$ is convex.

We have
\begin{flalign}
\nonumber
{\mathbf{E}_\mathcal{C}}[{\| {y - {\hat{z}_{i,t + 1}}} \|^2}] - {\mathbf{E}_\mathcal{C}}[{\| {y - {z_{i,t + 1}}} \|^2}] & = {\mathbf{E}_\mathcal{C}}[\langle {y - {{\hat z}_{i,t + 1}} + y - {z_{i,t + 1}},y - {{\hat z}_{i,t + 1}} - y + {z_{i,t + 1}}} \rangle] \\
\nonumber
&\le {\mathbf{E}_\mathcal{C}}[\| {y - {\hat{z}_{i,t + 1}} + y - {z_{i,t + 1}}} \|\| {{z_{i,t + 1}} - {\hat{z}_{i,t + 1}}} \|] \\
&\le 4R( \mathbb{X} ){\mathbf{E}_\mathcal{C}}[\| {{z_{i,t + 1}} - {\hat{z}_{i,t + 1}}} \|], \label{Lemma4-proof-eq7}
\end{flalign}
where the first inequality holds due to the Cauchy–Schwarz inequality; and the last inequality holds due to \eqref{ass1-eq1}.

We have
\begin{flalign}
\nonumber
{( {{\mathbf{E}_\mathcal{C}}[ {\| {{z_{i,t + 1}} - {\hat{z}_{i,t + 1}}} \|} ]} )^2} &\le {{\tilde p}^2}{( {{\mathbf{E}_\mathcal{C}}[ {\| {{z_{i,t + 1}} - {\hat{z}_{i,t + 1}}} \|_d} ]} )^2} \\
\nonumber
&\le {{\tilde p}^2}{\mathbf{E}_\mathcal{C}}[ {{\| {{z_{i,t + 1}} - {\hat{z}_{i,t + 1}}} \|_d^2}} ]  \\
\nonumber
& \le {{\tilde p}^2}{\mathbf{E}_\mathcal{C}}[{\| {{z_{i,t + 1}} - {\hat{z}_{i,t}} - {s_{t + 1}}\mathcal{C}( {( {{z_{i,t + 1}} - {\hat{z}_{i,t}}} )/{s_{t + 1}}} )} \|_d^2}] \\
\nonumber
& = {{\tilde p}^2}s_{t + 1}^2{\mathbf{E}_\mathcal{C}}[{\| {( {{z_{i,t + 1}} - {\hat{z}_{i,t}}} )/{s_{t + 1}} - \mathcal{C}( {( {{z_{i,t + 1}} - {\hat{z}_{i,t}}} )/{s_{t + 1}}} )} \|_d^2}]\\
& \le {{\tilde p}^2}Cs_{t + 1}^2, \label{Lemma4-proof-eq8}
\end{flalign}
where the first inequality holds due to Lemma~1; the second inequality holds due to the Jensen’s inequality; the first equality holds due to \eqref{Algorithm1-eq1}; and the last inequality holds due to Assumption~5.

Summing \eqref{Lemma4-proof-eq2}--\eqref{Lemma4-proof-eq8} over $i \in [ n ]$, dividing by $n$, using $\sum\nolimits_{i = 1}^n {{{\left[ {{W_t}} \right]}_{ij}}}  = 1$, $\forall t \in {\mathbb{N}_ + }$, and rearranging terms yields \eqref{Lemma4-eq1}. 
\end{proof}

\begin{lemma}
Suppose Assumptions~1--2 and 4--5 hold. For all $i \in [ n ]$, let $\{ {{x_{i,t}}} \}$ be the sequences generated by Algorithm~1 with ${\gamma _t} = {\gamma _0}/{\alpha _t}$, where ${\gamma _0} \in ( {0,1/( {4G_2^2} )} ]$ is a constant. Then, for any $T \in {\mathbb{N}_ + }$,
\begin{subequations}
\begin{flalign}
\nonumber
&\frac{1}{n}\sum\limits_{i = 1}^n {\sum\limits_{t = 1}^T {{\mathbf{E}_\mathcal{C}}\Big[ {\langle {\nabla {f_{i,t}}( {{x_{i,t}}} ),{x_{i,t}} - y} \rangle  + \frac{{{{\| {{x_{i,t}} - {z_{i,t + 1}}} \|}^2}}}{{4{\alpha _t}}}} \Big]} } 
\le 2G_1^2\sum\limits_{t = 1}^T {{\alpha _t}}  + \frac{1}{n}\sum\limits_{i = 1}^n {\sum\limits_{t = 1}^T {{\mathbf{E}_\mathcal{C}}[ {{\Delta_{i,t}}( y )} ]} } \\
&\;\;\;\;\;\;\;\;\;\;\;\;\;\;\;\;\;\;\;\;\;\;\;\;\;\;\;\;\;\;\;\;\;\;\;\;\;\;\;\;\;\;\;\;\;\;\;\;\;\;\;\;\;\;\;\;\;\;\;\;\;\;\;\;\;\;\;\;\;\;\;\;\;\;\;\;\;\;\; + 2{\tilde p}\sqrt C R( \mathbb{X} )\sum\limits_{t = 1}^T {\frac{{{s_{t + 1}}}}{{{\alpha _t}}}}, \forall y \in {\mathcal{X}_T}, \label{Lemma5-eq1} \\
&\sum\limits_{i = 1}^n {\sum\limits_{t = 1}^T {\frac{1}{2}} } {\mathbf{E}_\mathcal{C}}\Big[ { {\frac{{v_{i,t + 1}^T{g_{i,t}}( {{x_{i,t}}} )}}{{{\gamma _t}}} + \frac{{{{\| {{x_{i,t}} - {z_{i,t + 1}}} \|}^2}}}{{2{\gamma _0}}}} } \Big] \le {\mathbf{E}_\mathcal{C}}[ {{\Lambda _T}( y )} ] + {{{\tilde \Lambda }_T}( y )}, \forall y \in {\mathcal{X}_T}, \label{Lemma5-eq2} \\
&\frac{1}{n}\sum\limits_{t = 1}^T {\sum\limits_{i = 1}^n {\sum\limits_{j = 1}^n {\mathbf{E}_\mathcal{C}}[{\| {{x_{i,t}} - {x_{j,t}}} \|}] } }  \le n{\vartheta _1} + {\tilde \vartheta _2}\sum\limits_{t = 1}^T {\sum\limits_{i = 1}^n {\mathbf{E}_\mathcal{C}}[{\| {\varepsilon _{i,t}^z} \|}] }, \label{Lemma5-eq3} \\
&\frac{1}{n}\sum\limits_{t = 1}^T {\sum\limits_{i = 1}^n {\sum\limits_{j = 1}^n {{\mathbf{E}_\mathcal{C}}[{{\| {{x_{i,t}} - {x_{j,t}}} \|}^2}]} } }  \le {{\tilde \vartheta }_3} + {{\tilde \vartheta }_4}\sum\limits_{t = 1}^T {\sum\limits_{i = 1}^n {{\mathbf{E}_\mathcal{C}}[{{\| {\varepsilon _{i,t}^z} \|}^2}]} }, \label{Lemma5-eq4} \\
&{\mathbf{E}_\mathcal{C}}[\| {{z_{i,t + 1}} - {x_{i,t}}} \|] \le {{G_1}{\alpha _t} + {G_2}{\gamma _0}{\mathbf{E}_\mathcal{C}}[\| {{{[ {{g_{i,t}}( {{x_{i,t}}} )} ]}_ + }} \|]}, \label{Lemma5-eq5}
\end{flalign}
\end{subequations}
where
\begin{flalign}
\nonumber
&{\Lambda _T}( y ) = \sum\limits_{i = 1}^n {\sum\limits_{t = 1}^T {\frac{{v_{i,t + 1}^T{g_{i,t}}( y )}}{{{\gamma _t}}}} }, \\
\nonumber
&{{\tilde \Lambda }_T}( y ) = \sum\limits_{i = 1}^n {\frac{{{{\| {y - {x_{i,1}}} \|}^2}}}{{2{\gamma _0}}}}  + 2n{G_1}R( \mathbb{X} )\sum\limits_{t = 1}^T {\frac{1}{{{\gamma _t}}}}  + 2n{\gamma _0}G_1^2\sum\limits_{t = 1}^T {\frac{1}{{\gamma _t^2}}}  + \frac{{2n{\tilde p}\sqrt C R( \mathbb{X} )}}{{{\gamma _0}}}\sum\limits_{t = 1}^T {{s_{t}}}, \\
\nonumber
&{\vartheta _1} = \frac{{2\tau }}{{\lambda ( {1 - \lambda } )}}\sum\limits_{i = 1}^n {\| {{\hat{z}_{i,1}}} \|}, {{\tilde \vartheta }_2} = \frac{{4 - 4\lambda  + 2n\tau }}{{1 - \lambda }}, \\
\nonumber
&{{\tilde \vartheta }_3} = \frac{{16n{\tau ^2}}}{{{\lambda ^2}( {1 - {\lambda ^2}} )}}{\Big( {\sum\limits_{i = 1}^n {\| {{\hat{z}_{i,1}}} \|} } \Big)^2}, {{\tilde \vartheta }_4} = \frac{{16{n^2}{\tau ^2}}}{{{{( {1 - \lambda } )}^2}}} + 32.
\end{flalign}
\end{lemma}
\begin{proof}
(i)
Since ${g_{i,t}}( y ) \le {\mathbf{0}_{{m_i}}}$, $\forall i \in [ n ]$, $\forall t \in {\mathbb{N}_ + }$ when $\forall y \in {\mathcal{X}_T}$, summing \eqref{Lemma4-eq1} over $t \in [ T ]$ gives
\begin{flalign}
\nonumber
&\;\;\;\;\;\frac{1}{n}\sum\limits_{i = 1}^n {\sum\limits_{t = 1}^T {{\mathbf{E}_\mathcal{C}}[ {\langle {\nabla {f_{i,t}}( {{x_{i,t}}} ),{x_{i,t}} - y} \rangle } ]} } \\
& \le \frac{1}{n}\sum\limits_{i = 1}^n {\sum\limits_{t = 1}^T {{\mathbf{E}_\mathcal{C}}\Big[ { - v_{i,t + 1}^T{g_{i,t}}( {{x_{i,t}}} ) + {\Delta _{i,t}}( y ) + \frac{1}{n}{{\tilde \Delta }_t}} \Big] + 2{\tilde p}\sqrt C R( \mathbb{X} )\sum\limits_{t = 1}^T {\frac{{{s_{t + 1}}}}{{{\alpha _t}}}} } }. \label{Lemma5-proof-eq1}
\end{flalign}
We have
\begin{flalign}
\nonumber
&\;\;\;\;\;\frac{1}{n}\sum\limits_{i = 1}^n {\sum\limits_{t = 1}^T {( {{G_1} + {G_2}\| {{v_{i,t + 1}}} \|} )} } \| {{x_{i,t}} - {z_{i,t + 1}}} \| \\
& \le \frac{1}{n}\sum\limits_{i = 1}^n {\sum\limits_{t = 1}^T {\Big( {2G_1^2{\alpha _t} + 2G_2^2{\alpha _t}{{\| {{v_{i,t + 1}}} \|}^2} + \frac{{{{\| {{x_{i,t}} - {z_{i,t + 1}}} \|}^2}}}{{4{\alpha _t}}}} \Big)} }. \label{Lemma5-proof-eq2}
\end{flalign}
From \eqref{Algorithm1-eq4}, for all $\forall t \in {\mathbb{N}_ + }$, we have
\begin{flalign}
\| {{v_{i,t + 1}}} \| = {\gamma _t}\| {{{[ {{g_{i,t}}( {{x_{i,t}}} )} ]}_ + }} \|. \label{Lemma5-proof-eq3}
\end{flalign}
From \eqref{Lemma5-proof-eq3}, \eqref{Algorithm1-eq4} and the fact that ${\varphi ^T}{[ \varphi  ]_ + } = {\| {{{[ \varphi  ]}_ + }} \|^2}$ for any vector $\varphi $, we have
\begin{flalign}
\nonumber
2G_2^2{\alpha _t}{\mathbf{E}_\mathcal{C}}[{\| {{v_{i,t + 1}}} \|^2}] - {\mathbf{E}_\mathcal{C}}[v_{i,t + 1}^T{g_{i,t}}( {{x_{i,t}}} )]
& = 2G_2^2{\alpha _t}\gamma _t^2{\mathbf{E}_\mathcal{C}}[{\| {{{[ {{g_{i,t}}( {{x_{i,t}}} )} ]}_ + }} \|^2}] - {\gamma _t}{\mathbf{E}_\mathcal{C}}[{\| {{{[ {{g_{i,t}}( {{x_{i,t}}} )} ]}_ + }} \|^2}] \\
& = ( {2G_2^2{\gamma _0} - 1} ){\gamma _t}{\mathbf{E}_\mathcal{C}}[{\| {{{[ {{g_{i,t}}( {{x_{i,t}}} )} ]}_ + }} \|^2}] \le 0, \label{Lemma5-proof-eq4}
\end{flalign}
where the last equality holds due to ${\gamma _t} = {\gamma _0}/{\alpha _t}$; and the inequality holds due to ${\gamma _0} \in ( {0,1/( {4G_2^2} )} ]$.

Combining \eqref{Lemma5-proof-eq1}--\eqref{Lemma5-proof-eq2} and \eqref{Lemma5-proof-eq4} yields \eqref{Lemma5-eq1}.

(ii)
From the Cauchy--Schwarz inequality and \eqref{ass1-eq1}, we have
\begin{flalign}
\langle {\nabla {f_{i,t}}( {{x_{i,t}}} ),y - {x_{i,t}}} \rangle  \le \| {\nabla {f_{i,t}}( {{x_{i,t}}} )} \|\| {y - {x_{i,t}}} \| \le 2{G_1}R( \mathbb{X} ), \forall y \in \mathbb{X}. \label{Lemma5-proof-ii-eq1}
\end{flalign}
Dividing \eqref{Lemma4-eq1} by ${\gamma _t}$, using \eqref{Lemma5-proof-ii-eq1}, and summing over $t \in [ T ]$ yields
\begin{flalign}
\nonumber
&\;\;\;\;\;\sum\limits_{i = 1}^n {\sum\limits_{t = 1}^T {\frac{{{\mathbf{E}_\mathcal{C}}[ {v_{i,t + 1}^T{g_{i,t}}( {{x_{i,t}}} )} ]}}{{{\gamma _t}}}} } \\
& \le {\mathbf{E}_\mathcal{C}}[ {{\Lambda _T}( y )} ] + {2n{G_1}R( \mathbb{X} )}\sum\limits_{t = 1}^T {\frac{1}{{{\gamma _t}}}}  + \sum\limits_{i = 1}^n {\sum\limits_{t = 1}^T {\frac{{{\mathbf{E}_\mathcal{C}}[ {{\Delta _{i,t}}( y )} ]}}{{{\gamma _t}}} + } } \sum\limits_{t = 1}^T {\frac{{{\mathbf{E}_\mathcal{C}}[ {{{\tilde \Delta }_t}} ]}}{{{\gamma _t}}}}  + \frac{{2n{\tilde p}\sqrt C R( \mathbb{X} )}}{{{\gamma _0}}}\sum\limits_{t = 1}^T {{s_{t + 1}}}. \label{Lemma5-proof-ii-eq2}
\end{flalign}
We have
\begin{flalign}
\nonumber
&\;\;\;\;\;\sum\limits_{i = 1}^n {\sum\limits_{t = 1}^T {\frac{{( {{G_1} + {G_2}\| {{v_{i,t + 1}}} \|} )\| {{x_{i,t}} - {z_{i,t + 1}}} \|}}{{{\gamma _t}}}} } \\
& \le \sum\limits_{i = 1}^n {\sum\limits_{t = 1}^T {( {\frac{{2{\gamma _0}G_1^2}}{{\gamma _t^2}} + \frac{{2{\gamma _0}G_2^2{{\| {{v_{i,t + 1}}} \|}^2}}}{{\gamma _t^2}} + \frac{{{{\| {{x_{i,t}} - {z_{i,t + 1}}} \|}^2}}}{{4{\gamma _0}}}} )} }. \label{Lemma5-proof-ii-eq3}
\end{flalign}
From ${\gamma _t} = {\gamma _0}/{\alpha _t}$, we have
\begin{flalign}
\sum\limits_{t = 1}^T {\frac{{{\mathbf{E}_\mathcal{C}}[ {{\Delta _{i,t}}( y )} ]}}{{{\gamma _t}}}}  = \frac{1}{{2{\gamma _0}}}\sum\limits_{t = 1}^T {{\mathbf{E}_\mathcal{C}}[ {{{\| {y - {x_{i,t}}} \|}^2} - {{\| {y - {x_{i,t + 1}}} \|}^2}} ]}  \le \frac{{{{\| {y - {x_{i,1}}} \|}^2}}}{{2{\gamma _0}}}. \label{Lemma5-proof-ii-eq4}
\end{flalign}
From \eqref{Lemma5-proof-eq3}, \eqref{Algorithm1-eq4} and the fact that ${\varphi ^T}{[ \varphi  ]_ + } = {\| {{{[ \varphi  ]}_ + }} \|^2}$ for any vector $\varphi $, we have
\begin{flalign}
\nonumber
\frac{{2{\gamma _0}G_2^2{\mathbf{E}_\mathcal{C}}[ {{{\| {{v_{i,t + 1}}} \|}^2}} ]}}{{\gamma _t^2}} - \frac{{{\mathbf{E}_\mathcal{C}}[ {v_{i,t + 1}^T{g_{i,t}}( {{x_{i,t}}} )} ]}}{{2{\gamma _t}}} & = 2{\gamma _0}G_2^2{\mathbf{E}_\mathcal{C}}[ {{{\| {{{[ {{g_{i,t}}( {{x_{i,t}}} )} ]}_ + }} \|}^2}} ] - \frac{{{\mathbf{E}_\mathcal{C}}[ {{{\| {{{[ {{g_{i,t}}( {{x_{i,t}}} )} ]}_ + }} \|}^2}} ]}}{2} \\
& = ( {2{\gamma _0}G_2^2 - \frac{1}{2}} ){\mathbf{E}_\mathcal{C}}[ {{{\| {{{[ {{g_{i,t}}( {{x_{i,t}}} )} ]}_ + }} \|}^2}} ] \le 0, \label{Lemma5-proof-ii-eq5}
\end{flalign}
where the inequality holds due to ${\gamma _0} \in ( {0,1/( {4G_2^2} )} ]$.

Combining \eqref{Lemma5-proof-ii-eq2}--\eqref{Lemma5-proof-ii-eq5}, and noting that $\{ {{s_t}} \}$ is nonincreasing, we have \eqref{Lemma5-eq2}.

(iii)
From \eqref{Algorithm1-eq3} and $\sum\nolimits_{i = 1}^n {{{[ {{W_t}} ]}_{ij}}}  = \sum\nolimits_{j = 1}^n {{{[ {{W_t}} ]}_{ij}}}  = 1$, we have
\begin{flalign}
\nonumber
\frac{1}{n}\sum\limits_{i = 1}^n {\sum\limits_{j = 1}^n {{\mathbf{E}_\mathcal{C}}[ {\| {{x_{i,t}} - {x_{j,t}}} \|} ]} } & = \frac{1}{n}\sum\limits_{i = 1}^n {\sum\limits_{j = 1}^n {{\mathbf{E}_\mathcal{C}}\Big[ {\Big\| {\sum\limits_{k = 1}^n {{{[ {{W_t}} ]}_{ik}}{{\hat z}_{k,t}}}  - {{\bar z}_t} + {{\bar z}_t} - \sum\limits_{k = 1}^n {{{[ {{W_t}} ]}_{jk}}{{\hat z}_{k,t}}} } \Big\|} \Big]} } \\
\nonumber
& \le \frac{1}{n}\sum\limits_{i = 1}^n {\sum\limits_{j = 1}^n {{\mathbf{E}_\mathcal{C}}\Big[ {\Big\| {\sum\limits_{k = 1}^n {{{[ {{W_t}} ]}_{ik}}{{\hat z}_{k,t}}}  - {{\bar z}_t}} \Big\|} \Big]} }  + \frac{1}{n}\sum\limits_{i = 1}^n {\sum\limits_{j = 1}^n {{\mathbf{E}_\mathcal{C}}\Big[ {\Big\| {{{\bar z}_t} - \sum\limits_{k = 1}^n {{{[ {{W_t}} ]}_{jk}}{{\hat z}_{k,t}}} } \Big\|} \Big]} } \\
\nonumber
& = 2\sum\limits_{i = 1}^n {{\mathbf{E}_\mathcal{C}}\Big[ {\Big\| {\sum\limits_{j = 1}^n {{{[ {{W_t}} ]}_{ij}}{{\hat z}_{j,t}}}  - {{\bar z}_t}} \Big\|} \Big]}  = 2\sum\limits_{i = 1}^n {{\mathbf{E}_\mathcal{C}}\Big[ {\Big\| {\sum\limits_{j = 1}^n {{{[ {{W_t}} ]}_{ij}}( {{{\hat z}_{j,t}} - {{\bar z}_t}} )} } \Big\|} \Big]} \\
& \le 2\sum\limits_{i = 1}^n {\sum\limits_{j = 1}^n {{{[ {{W_t}} ]}_{ij}}} {\mathbf{E}_\mathcal{C}}[ {\| {{{\hat z}_{j,t}} - {{\bar z}_t}} \|} ]}  = 2\sum\limits_{i = 1}^n {{\mathbf{E}_\mathcal{C}}[ {\| {{{\hat z}_{i,t}} - {{\bar z}_t}} \|} ]}. \label{Lemma5-proof-iii-eq1}
\end{flalign}
We have
\begin{flalign}
\sum\limits_{t = 3}^T {\sum\limits_{s = 1}^{t - 2} {{\lambda ^{t - s - 2}}} } \sum\limits_{j = 1}^n {{\mathbf{E}_\mathcal{C}}[ {\| {\varepsilon _{j,s}^z} \|} ]}  = \sum\limits_{t = 1}^{T - 2} {\sum\limits_{j = 1}^n {{\mathbf{E}_\mathcal{C}}[ {\| {\varepsilon _{j,t}^z} \|} ]} \sum\limits_{s = 0}^{T - t - 2} {{\lambda ^s}} }  \le \frac{1}{{1 - \lambda }}\sum\limits_{t = 1}^{T - 2} {\sum\limits_{j = 1}^n {{\mathbf{E}_\mathcal{C}}[ {\| {\varepsilon _{j,t}^z} \|} ]} }. \label{Lemma5-proof-iii-eq2}
\end{flalign}
From \eqref{Lemma5-proof-iii-eq1}, \eqref{Lemma3-eq1}, \eqref{Lemma5-proof-iii-eq2}, we have
\begin{flalign}
\nonumber
&\;\;\;\;\;\frac{1}{n}\sum\limits_{t = 1}^T {\sum\limits_{i = 1}^n {\sum\limits_{j = 1}^n {{\mathbf{E}_\mathcal{C}}[ {\| {{x_{i,t}} - {x_{j,t}}} \|} ]} } } \\
\nonumber
& \le 2\sum\limits_{t = 1}^T {\sum\limits_{i = 1}^n {{\mathbf{E}_\mathcal{C}}[ {\| {{{\hat z}_{i,t}} - {{\bar z}_t}} \|} ]} } \\
\nonumber
& \le 2\tau \sum\limits_{t = 1}^T {{\lambda ^{t - 2}}\sum\limits_{i = 1}^n {\sum\limits_{j = 1}^n {\| {{\hat{z}_{j,1}}} \|} } }  + \frac{2}{n}\sum\limits_{t = 2}^T {\sum\limits_{i = 1}^n {\sum\limits_{j = 1}^n {{\mathbf{E}_\mathcal{C}}[ {\| {\varepsilon _{j,t - 1}^z} \|} ]} } }  + 2\sum\limits_{t = 2}^T {\sum\limits_{i = 1}^n {{\mathbf{E}_\mathcal{C}}[ {\| {\varepsilon _{i,t - 1}^z} \|} ]} }  \\
\nonumber
&\;\;+ 2\tau \sum\limits_{t = 3}^T {\sum\limits_{i = 1}^n {\sum\limits_{s = 1}^{t - 2} {{\lambda ^{t - s - 2}}\sum\limits_{j = 1}^n {{\mathbf{E}_\mathcal{C}}[ {\| {\varepsilon _{j,s}^z} \|} ]} } } } \\
\nonumber
& \le \frac{{2\tau }}{{\lambda ( {1 - \lambda } )}}\sum\limits_{i = 1}^n {\sum\limits_{j = 1}^n {\| {{\hat{z}_{j,1}}} \|} }  + 4\sum\limits_{t = 2}^T {\sum\limits_{i = 1}^n {{\mathbf{E}_\mathcal{C}}[ {\| {\varepsilon _{i,t - 1}^z} \|} ]} }  + \frac{{2n\tau }}{{1 - \lambda }}\sum\limits_{t = 1}^{T - 2} {\sum\limits_{j = 1}^n {{\mathbf{E}_\mathcal{C}}[ {\| {\varepsilon _{j,t}^z} \|} ]} } \\
\nonumber
& = \frac{{2\tau }}{{\lambda ( {1 - \lambda } )}}\sum\limits_{i = 1}^n {\sum\limits_{j = 1}^n {\| {{\hat{z}_{j,1}}} \|} }  + 4\sum\limits_{t = 1}^{T - 1} {\sum\limits_{i = 1}^n {{\mathbf{E}_\mathcal{C}}[ {\| {\varepsilon _{i,t}^z} \|} ]} }  + \frac{{2n\tau }}{{1 - \lambda }}\sum\limits_{t = 1}^{T - 2} {\sum\limits_{i = 1}^n {{\mathbf{E}_\mathcal{C}}[ {\| {\varepsilon _{i,t}^z} \|} ]} } \\
& \le n{\vartheta _1} + {{\tilde \vartheta }_2}\sum\limits_{t = 1}^T {\sum\limits_{i = 1}^n {{\mathbf{E}_\mathcal{C}}[ {\| {\varepsilon _{i,t}^z} \|} ]} }. \label{Lemma5-proof-iii-eq3}
\end{flalign}
Therefore, from \eqref{Lemma5-proof-iii-eq3}, we know that \eqref{Lemma5-eq3} holds.

(iv)
Similar to the way to get \eqref{Lemma5-proof-iii-eq1}, from \eqref{Algorithm1-eq3} and $\sum\nolimits_{i = 1}^n {{{[ {{W_t}} ]}_{ij}}}  = \sum\nolimits_{j = 1}^n {{{[ {{W_t}} ]}_{ij}}}  = 1$, and ${\|  \cdot  \|^2}$ is convex, we have
\begin{flalign}
\frac{1}{n}\sum\limits_{i = 1}^n {\sum\limits_{j = 1}^n {{\mathbf{E}_\mathcal{C}}[{{\| {{x_{i,t}} - {x_{j,t}}} \|}^2}]} }  \le 4\sum\limits_{i = 1}^n {{\mathbf{E}_\mathcal{C}}[{{\| {{{\hat z}_{i,t}} - {{\bar z}_t}} \|}^2}]}. \label{Lemma5-proof-iv-eq1}
\end{flalign}
From \eqref{Lemma3-eq1}, we have
\begin{flalign}
\nonumber
&\;\;\;\;\;4\sum\limits_{t = 1}^T {\sum\limits_{i = 1}^n {{\mathbf{E}_\mathcal{C}}[{{\| {{{\hat z}_{i,t}} - {{\bar z}_t}} \|}^2}]} } \\
\nonumber
& \le 4{\sum\limits_{t = 1}^T \sum\limits_{i = 1}^n {\mathbf{E}_\mathcal{C}}\Big[\Big( {\tau {\lambda ^{t - 2}}\sum\limits_{j = 1}^n {\| {{{\hat z}_{j,1}}} \|}  + \| {\varepsilon _{i,t - 1}^z} \| + \frac{1}{n}\sum\limits_{j = 1}^n {\| {\varepsilon _{j,t - 1}^z} \|}  + \tau \sum\limits_{s = 1}^{t - 2} {{\lambda ^{t - s - 2}}\sum\limits_{j = 1}^n {\| {\varepsilon _{j,s}^z} \|} } } \Big)^2}\Big] \\
\nonumber
& \le 16\sum\limits_{i = 1}^n \sum\limits_{t = 1}^T {\mathbf{E}_\mathcal{C}}\Big[{\Big( {{\Big( {\tau {\lambda ^{t - 2}}\sum\limits_{j = 1}^n {\| {{{\hat z}_{j,1}}} \|} } \Big)^2} + {{\| {\varepsilon _{i,t - 1}^z} \|}^2} + {\Big( {\frac{1}{n}\sum\limits_{j = 1}^n {\| {\varepsilon _{j,t - 1}^z} \|} } \Big)^2} + {\Big( {\tau \sum\limits_{s = 1}^{t - 2} {{\lambda ^{t - s - 2}}\sum\limits_{j = 1}^n {\| {\varepsilon _{j,s}^z} \|} } } \Big)^2}} \Big)} \Big]  \\
\nonumber
& \le 16\sum\limits_{i = 1}^n \sum\limits_{t = 1}^T {\mathbf{E}_\mathcal{C}}\Big[\Big( {\Big( {\tau {\lambda ^{t - 2}}\sum\limits_{j = 1}^n {\| {{{\hat z}_{j,1}}} \|} } \Big)^2} + {{\| {\varepsilon _{i,t - 1}^z} \|}^2} + \frac{1}{n}\sum\limits_{j = 1}^n {{{\| {\varepsilon _{j,t - 1}^z} \|}^2}}  \\
\nonumber
& \;\;\;\;\;\;\;\;\;\;\;\;\;\;\;\;\;\;\;\;\;\;\;\;\;\;\;\;\;\;\; \;\;\;\;\;\;\;\;\;\;\;\;\;\;\;\;\;\;\;\;\;\;\;\;\;\;\;\;\;\;\;\;\;\;\;\;\;\;\;\;\;\;\;\; + {\tau ^2}\sum\limits_{s = 1}^{t - 2} {{\lambda ^{t - s - 2}}\sum\limits_{s = 1}^{t - 2} {{\lambda ^{t - s - 2}}} {\Big( {\sum\limits_{j = 1}^n {\| {\varepsilon _{j,s}^z} \|} } \Big)^2}}  \Big)\Big] \\
\nonumber
& \le 16\sum\limits_{i = 1}^n {\sum\limits_{t = 1}^T {\mathbf{E}_\mathcal{C}}\Big[{\Big( {{\Big( {\tau {\lambda ^{t - 2}}\sum\limits_{j = 1}^n {\| {{{\hat z}_{j,1}}} \|} } \Big)^2} + 2{{\| {\varepsilon _{i,t - 1}^z} \|}^2} + \frac{{n{\tau ^2}}}{{1 - \lambda }}\sum\limits_{s = 1}^{t - 2} {{\lambda ^{t - s - 2}}} \sum\limits_{j = 1}^n {{{\| {\varepsilon _{j,s}^z} \|}^2}} } \Big)} }\Big] \\
\nonumber
& = 16\sum\limits_{i = 1}^n {\sum\limits_{t = 1}^T {\mathbf{E}_\mathcal{C}}\Big[{\Big( {{\Big( {\tau {\lambda ^{t - 2}}\sum\limits_{j = 1}^n {\| {{{\hat z}_{j,1}}} \|} } \Big)^2} + 2{{\| {\varepsilon _{i,t - 1}^z} \|}^2}} \Big) + \frac{{16{n^2}{\tau ^2}}}{{1 - \lambda }}\sum\limits_{j = 1}^n {\sum\limits_{t = 1}^{T - 2} {{{\| {\varepsilon _{j,t}^z} \|}^2}} } \sum\limits_{s = 0}^{T - t - 2} {{\lambda ^s}} } }\Big]  \\
& \le {{\tilde \vartheta }_3} + {{\tilde \vartheta }_4}\sum\limits_{t = 1}^T {\sum\limits_{i = 1}^n {{\mathbf{E}_\mathcal{C}}[{{\| {\varepsilon _{i,t}^z} \|}^2}]} }, \label{Lemma5-proof-iv-eq2}
\end{flalign}
where the third inequality holds due to the H\"{o}lder’s inequality.

Therefore, from \eqref{Lemma5-proof-iv-eq1}--\eqref{Lemma5-proof-iv-eq2}, we know that \eqref{Lemma5-eq4} holds.

(v)
Applying \eqref{Lemma2-eq2} to the update \eqref{Algorithm1-eq6} gives
\begin{flalign}
\nonumber
{\mathbf{E}_\mathcal{C}}[\| {{z_{i,t + 1}} - {x_{i,t}}} \|] \le {{\alpha _t}{\mathbf{E}_\mathcal{C}}[\| {{\omega _{i,t + 1}}} \|]} &= {{\alpha _t}{\mathbf{E}_\mathcal{C}}[\| {\nabla {f_{i,t}}( {{x_{i,t}}} ) + {{{\big( {\nabla {g_{i,t}}( {{x_{i,t}}} )} \big)}^T}}{v_{i,t + 1}}} \|]} \\
\nonumber
& \le {{\alpha _t}}{\mathbf{E}_\mathcal{C}}[ {\| {\nabla {f_{i,t}}( {{x_{i,t}}} )} \| + \| {{{{\big( {\nabla {g_{i,t}}( {{x_{i,t}}} )} \big)}^T}}{v_{i,t + 1}}} \|}] \\
\nonumber
& \le {{\alpha _t}}{\mathbf{E}_\mathcal{C}}[{{G_1} + {G_2}{\gamma _t}\| {{{[ {{g_{i,t}}( {{x_{i,t}}} )} ]}_ + }} \|}] \\
& = {\mathbf{E}_\mathcal{C}}[{{G_1}{\alpha _t} + {G_2}{\gamma _0}\| {{{[ {{g_{i,t}}( {{x_{i,t}}} )} ]}_ + }} \|}], \label{Lemma5-proof-v-eq1}
\end{flalign}
where the first equality holds due to \eqref{Algorithm1-eq5}; the last inequality holds due to \eqref{ass4-eq1a}, \eqref{ass4-eq1b}, and \eqref{Lemma5-proof-eq3}; and last equality holds due to ${\gamma _t} = {\gamma _0}/{\alpha _t}$.

Therefore, from \eqref{Lemma5-proof-v-eq1}, we know that \eqref{Lemma5-eq5} holds.
\end{proof}
\begin{lemma}
Under the same conditions as stated in Lemma~6, and supposing that Assumption~3 holds, for any $T \in {\mathbb{N}_ + }$ and any $y \in {\mathcal{X}_T}$, it holds that
\begin{subequations}
\begin{flalign}
\nonumber
&\frac{1}{n}\sum\limits_{i = 1}^n {\sum\limits_{t = 1}^T {{\mathbf{E}_\mathcal{C}}[ {\langle {\nabla {f_t}( {{x_{i,t}}} ),{x_{i,t}} - y} \rangle } ]} }
\le \hat \vartheta + {\vartheta _2}\sum\limits_{t = 1}^T {{\alpha _t}}  + \tilde p\sqrt C {{\tilde \vartheta }_1}\sum\limits_{t = 1}^T {{s_t}} \\
&\;\;\;\;\;\;\;\;\;\;\;\;\;\;\;\;\;\;\;\;\;\;\;\;\;\;\;\;\;\;\;\;\;\;\;\;\;\;\;\;\;\;\;\;\;\;\;\;\;\;\; + 2{\tilde p}\sqrt C R( \mathbb{X} )\sum\limits_{t = 1}^T {\frac{{{s_t}}}{{{\alpha _t}}}}  + \frac{1}{n}\sum\limits_{i = 1}^n {\sum\limits_{t = 1}^T {{\mathbf{E}_\mathcal{C}}[ {{\Delta _{i,t}}( y )} ]} }, \label{Lemma6-eq1} \\
&\frac{1}{n}\sum\limits_{i = 1}^n {\sum\limits_{t = 1}^T {{\mathbf{E}_\mathcal{C}}[ {\| {{{[ {{g_t}( {{x_{i,t}}} )} ]}_ + }} \|} ]} }  \le \sqrt {{\vartheta _3}T + {\vartheta _4}T{{\tilde \Lambda }_T}( y ) + 4n{\tilde p^2}CG_2^2{{\tilde \vartheta }_4}T\sum\limits_{t = 1}^T {s_t^2} }, \label{Lemma6-eq2} \\
\nonumber
&\frac{1}{n}\sum\limits_{i = 1}^n {\sum\limits_{t = 1}^T {{\mathbf{E}_\mathcal{C}}[ {\| {{{[ {{g_t}( {{x_{i,t}}} )} ]}_ + }} \|} ]} }  \le n{G_2}{\vartheta _1} + {\vartheta _5}\sum\limits_{t = 1}^T {{\alpha _t}}  + {\vartheta _6}\sum\limits_{i = 1}^n {\sum\limits_{t = 1}^T {{\mathbf{E}_\mathcal{C}}[ {\| {{{[ {{g_{i,t}}( {{x_{i,t}}} )} ]}_ + }} \|} ]} } \\
&\;\;\;\;\;\;\;\;\;\;\;\;\;\;\;\;\;\;\;\;\;\;\;\;\;\;\;\;\;\;\;\;\;\;\;\;\;\;\;\; + n{\tilde p}\sqrt C {G_2}{{\tilde \vartheta }_2}\sum\limits_{t = 1}^T {{s_t}}, \label{Lemma6-eq3}
\end{flalign}
\end{subequations}
where
\begin{flalign}
\nonumber
&\hat \vartheta  = 2R( \mathbb{X} )L{\vartheta _1} + {G_1}{\vartheta _1}, {{\tilde \vartheta }_1} = 2R( \mathbb{X} )L{{\tilde \vartheta }_2} + {G_1}{{\tilde \vartheta }_2},
{\vartheta _2} = 2G_1^2 + \tilde \vartheta _1^2, {\vartheta _3} = 2G_2^2{{\tilde \vartheta }_3},  \\
\nonumber
&{\vartheta _4} = \frac{{4\max \{ {1,2G_2^2{{\tilde \vartheta }_4}} \}}}{{\min \{ {1,\frac{1}{{2{\gamma _0}}}} \}}}, {\vartheta _5} = {n{G_1}{G_2}{{\tilde \vartheta }_2}}, {\vartheta _6} = {1  + G_2^2{\gamma _0}{{\tilde \vartheta }_2}}.
\end{flalign}
\end{lemma}
\begin{proof}
(i)
From ${f_t}( x ) = \frac{1}{n}\sum\nolimits_{j = 1}^n {{f_{j,t}}( x )}$, we have
\begin{flalign}
\nabla {f_t}( x ) = \frac{1}{n}\sum\nolimits_{j = 1}^n {\nabla {f_{j,t}}( x )}. \label{Lemma6-proof-i-eq1}
\end{flalign}
From \eqref{Lemma6-proof-i-eq1}, we have
\begin{flalign}
\nonumber
\sum\limits_{i = 1}^n {{\mathbf{E}_\mathcal{C}}[ {\nabla {f_t}( {{x_{i,t}}} )} ]}  &= \frac{1}{n}\sum\limits_{i = 1}^n {\sum\limits_{j = 1}^n {{\mathbf{E}_\mathcal{C}}[ {\nabla {f_{j,t}}( {{x_{i,t}}} )} ]} } \\
\nonumber
& = \frac{1}{n}\sum\limits_{i = 1}^n {\sum\limits_{j = 1}^n {{\mathbf{E}_\mathcal{C}}[ {\nabla {f_{j,t}}( {{x_{j,t}}} )} ]} }  + \frac{1}{n}\sum\limits_{i = 1}^n {\sum\limits_{j = 1}^n {{\mathbf{E}_\mathcal{C}}[ {\nabla {f_{j,t}}( {{x_{i,t}}} ) - \nabla {f_{j,t}}( {{x_{j,t}}} )} ]} } \\
& = \sum\limits_{i = 1}^n {{\mathbf{E}_\mathcal{C}}[ {\nabla {f_{i,t}}( {{x_{i,t}}} )} ]}  + \frac{1}{n}\sum\limits_{i = 1}^n {\sum\limits_{j = 1}^n {{\mathbf{E}_\mathcal{C}}[ {\nabla {f_{j,t}}( {{x_{i,t}}} ) - \nabla {f_{j,t}}( {{x_{j,t}}} )} ]} }.
\label{Lemma6-proof-i-eq2}
\end{flalign}
From \eqref{Lemma6-proof-i-eq2}, \eqref{ass5-eq1} and \eqref{ass1-eq1}, we have
\begin{flalign}
\nonumber
&\;\;\;\;\;\frac{1}{n}\sum\limits_{i = 1}^n {\sum\limits_{t = 1}^T {{\mathbf{E}_\mathcal{C}}[ {\langle {\nabla {f_t}( {{x_{i,t}}} ),{x_{i,t}} - y} \rangle } ]} } \\
\nonumber
& \le \frac{1}{n}\sum\limits_{i = 1}^n {\sum\limits_{t = 1}^T {{\mathbf{E}_\mathcal{C}}[ {\langle {\nabla {f_t}( {{x_{i,t}}} ),{{\bar z}_t} - y} \rangle } ]} }  + \frac{1}{{{n^2}}}\sum\limits_{i = 1}^n {\sum\limits_{j = 1}^n {\sum\limits_{t = 1}^T {{\mathbf{E}_\mathcal{C}}[ {\| {\nabla {f_{j,t}}( {{x_{i,t}}} )} \|\| {{x_{i,t}} - {{\bar z}_t}} \|} ]} } } \\
\nonumber
& \le \frac{1}{n}\sum\limits_{i = 1}^n {\sum\limits_{t = 1}^T {{\mathbf{E}_\mathcal{C}}[ {\langle {\nabla {f_{i,t}}( {{x_{i,t}}} ),{{\bar z}_t} - y} \rangle } ]} }  + \frac{{2R( \mathbb{X} )L}}{{{n^2}}}\sum\limits_{i = 1}^n {\sum\limits_{j = 1}^n {\sum\limits_{t = 1}^T {{\mathbf{E}_\mathcal{C}}[ {\| {{x_{i,t}} - {x_{j,t}}} \|} ]} } } \\
\nonumber
&\;\; + \frac{{{G_1}}}{n}\sum\limits_{i = 1}^n {\sum\limits_{t = 1}^T {{\mathbf{E}_\mathcal{C}}[ {\| {{x_{i,t}} - {{\bar z}_t}} \|} ]} } \\
\nonumber
& \le \frac{1}{n}\sum\limits_{i = 1}^n {\sum\limits_{t = 1}^T {{\mathbf{E}_\mathcal{C}}[ {\langle {\nabla {f_{i,t}}( {{x_{i,t}}} ),{x_{i,t}} - y} \rangle } ]} }  + \frac{{2R( \mathbb{X} )L}}{{{n^2}}}\sum\limits_{i = 1}^n {\sum\limits_{j = 1}^n {\sum\limits_{t = 1}^T {{\mathbf{E}_\mathcal{C}}[ {\| {{x_{i,t}} - {x_{j,t}}} \|} ]} } } \\
&\;\; + \frac{{2{G_1}}}{n}\sum\limits_{i = 1}^n {\sum\limits_{t = 1}^T {{\mathbf{E}_\mathcal{C}}[ {\| {{x_{i,t}} - {{\bar z}_t}} \|} ]} }. 
\label{Lemma6-proof-i-eq3}
\end{flalign}

From \eqref{Lemma5-eq3}, we have
\begin{flalign}
\frac{1}{n^2}\sum\limits_{i = 1}^n {\sum\limits_{j = 1}^n {\sum\limits_{t = 1}^T {{\mathbf{E}_\mathcal{C}}[ {2R( \mathbb{X} )L\| {{x_{i,t}} - {x_{j,t}}} \|} ]} } }
 \le 2R( \mathbb{X} )L{\vartheta _1} + \frac{1}{n}\sum\limits_{i = 1}^n {\sum\limits_{t = 1}^T {{\mathbf{E}_\mathcal{C}}[ {2R( \mathbb{X} )L{{\tilde \vartheta }_2}\| {\varepsilon _{i,t}^z} \|} ]} }.
\label{Lemma6-proof-i-eq4}
\end{flalign}

From \eqref{Lemma5-proof-iii-eq3}, we have
\begin{flalign}
\frac{{2{G_1}}}{n}\sum\limits_{i = 1}^n {\sum\limits_{t = 1}^T {{\mathbf{E}_\mathcal{C}}[ {\| {{{\bar z}_{i,t}} - {{\bar z}_t}} \|} ]} }  \le {G_1}{\vartheta _1} + \frac{{{G_1}{{\tilde \vartheta }_2}}}{n}\sum\limits_{t = 1}^T {\sum\limits_{i = 1}^n {{\mathbf{E}_\mathcal{C}}[ {\| {\varepsilon _{i,t}^z} \|} ]} }.
\label{Lemma6-proof-i-eq5}
\end{flalign}

From \eqref{Lemma4-proof-eq8}, we have
\begin{flalign}
\nonumber
&\;\;\;\;\;\frac{1}{n}\sum\limits_{t = 1}^T {\sum\limits_{i = 1}^n {{\mathbf{E}_\mathcal{C}}[ {{{\tilde \vartheta }_1}\| {\varepsilon _{i,t}^z} \|} ]} } \\
\nonumber
& = \frac{1}{n}\sum\limits_{t = 1}^T {\sum\limits_{i = 1}^n {{\mathbf{E}_\mathcal{C}}[ {{{\tilde \vartheta }_1}\| {{z_{i,t + 1}} - {{\hat z}_{i,t + 1}}} \|} ]} }  + \frac{1}{n}\sum\limits_{t = 1}^T {\sum\limits_{i = 1}^n {{\mathbf{E}_\mathcal{C}}[ {{{\tilde \vartheta }_1}\| {{z_{i,t + 1}} - {x_{i,t}}} \|} ]} } \\
& \le \tilde p\sqrt C {{\tilde \vartheta }_1}\sum\limits_{t = 1}^T {{s_{t + 1}}}  + \frac{1}{n}\sum\limits_{t = 1}^T {\sum\limits_{i = 1}^n {{\mathbf{E}_\mathcal{C}}\Big[ {\tilde \vartheta _1^2{\alpha _t} + \frac{{{{\| {{z_{i,t + 1}} - {x_{i,t}}} \|}^2}}}{{4{\alpha _t}}}} \Big]} }.
\label{Lemma6-proof-i-eq6}
\end{flalign}

Combining \eqref{Lemma6-proof-i-eq3}--\eqref{Lemma6-proof-i-eq4} and \eqref{Lemma5-eq1}, and noting that $\{ {{s_t}} \}$ is nonincreasing, we know that \eqref{Lemma6-eq1} holds.

(ii)
We have
\begin{flalign}
\nonumber
{\mathbf{E}_\mathcal{C}}[ {{{\| {{{[ {{g_{i,t}}( {{x_{i,t}}} )} ]}_ + }} \|}^2}} ] &= {\mathbf{E}_\mathcal{C}}[ {{{\| {{{[ {{g_{i,t}}( {{x_{i,t}}} )} ]}_ + } - {{[ {{g_{i,t}}( {{x_{j,t}}} )} ]}_ + } + {{[ {{g_{i,t}}( {{x_{j,t}}} )} ]}_ + }} \|}^2}} ] \\
\nonumber
& \ge \frac{1}{2}{\mathbf{E}_\mathcal{C}}[ {{{\| {{{[ {{g_{i,t}}( {{x_{j,t}}} )} ]}_ + }} \|}^2}} ] - {\mathbf{E}_\mathcal{C}}[ {{{\| {{{[ {{g_{i,t}}( {{x_{i,t}}} )} ]}_ + } - {{[ {{g_{i,t}}( {{x_{j,t}}} )} ]}_ + }} \|}^2}} ] \\
\nonumber
& \ge \frac{1}{2}{\mathbf{E}_\mathcal{C}}[ {{{\| {{{[ {{g_{i,t}}( {{x_{j,t}}} )} ]}_ + }} \|}^2}} ] - {\mathbf{E}_\mathcal{C}}[ {{{\| {{g_{i,t}}( {{x_{i,t}}} ) - {g_{i,t}}( {{x_{j,t}}} )} \|}^2}} ] \\
& \ge \frac{1}{2}{\mathbf{E}_\mathcal{C}}[ {{{\| {{{[ {{g_{i,t}}( {{x_{j,t}}} )} ]}_ + }} \|}^2}} ] - G_2^2{\mathbf{E}_\mathcal{C}}[ {{{\| {{x_{i,t}} - {x_{j,t}}} \|}^2}},
\label{Lemma6-proof-ii-eq1}
\end{flalign}
where the second inequality holds due to the nonexpansive property of the projection ${[  \cdot  ]_ + }$; and the last inequality holds due to \eqref{ass4-eq2}.

From ${g_t}( x ) = {\rm{col}}( {{g_{1,t}}( x ), \cdot  \cdot  \cdot ,{g_{n,t}}( x )} )$, we have
\begin{flalign}
\sum\limits_{i = 1}^n {\sum\limits_{j = 1}^n {\sum\limits_{t = 1}^T {{\mathbf{E}_\mathcal{C}}[ {{{\| {{{[ {{g_{i,t}}( {{x_{j,t}}} )} ]}_ + }} \|}^2}} ]} } }  = \sum\limits_{i = 1}^n {\sum\limits_{t = 1}^T {{\mathbf{E}_\mathcal{C}}[ {{{\| {{{[ {{g_t}( {{x_{i,t}}} )} ]}_ + }} \|}^2}} ]} }.
\label{Lemma6-proof-ii-eq2}
\end{flalign}
From \eqref{Lemma6-proof-ii-eq1}--\eqref{Lemma6-proof-ii-eq2}, we have
\begin{flalign}
\nonumber
&\;\;\;\;\;\frac{1}{n}\sum\limits_{i = 1}^n {\sum\limits_{t = 1}^T {{\mathbf{E}_\mathcal{C}}[ {{{\| {{{[ {{g_t}( {{x_{i,t}}} )} ]}_ + }} \|}^2}} ]} }  \\
\nonumber
&\le \frac{2}{n}\sum\limits_{i = 1}^n {\sum\limits_{j = 1}^n {\sum\limits_{t = 1}^T {{\mathbf{E}_\mathcal{C}}[ {{{\| {{{[ {{g_{i,t}}( {{x_{i,t}}} )} ]}_ + }} \|}^2}} ]} } }  + \frac{{2G_2^2}}{n}\sum\limits_{i = 1}^n {\sum\limits_{j = 1}^n {\sum\limits_{t = 1}^T {{\mathbf{E}_\mathcal{C}}[ {{{\| {{x_{i,t}} - {x_{j,t}}} \|}^2}} ]} } } \\
\nonumber
& \le {\vartheta _3} + 2\sum\limits_{i = 1}^n {\sum\limits_{t = 1}^T {{\mathbf{E}_\mathcal{C}}[ {{{\| {{{[ {{g_{i,t}}( {{x_{i,t}}} )} ]}_ + }} \|}^2}} ]} }  + 2G_2^2{{\tilde \vartheta }_4}\sum\limits_{i = 1}^n {\sum\limits_{t = 1}^T {{\mathbf{E}_\mathcal{C}}[ {{{\| {\varepsilon _{i,t}^z} \|}^2}} ]} } \\
\nonumber
& \le {\vartheta _3} + 2\sum\limits_{i = 1}^n {\sum\limits_{t = 1}^T {{\mathbf{E}_\mathcal{C}}[ {{{\| {{{[ {{g_{i,t}}( {{x_{i,t}}} )} ]}_ + }} \|}^2}} ]} }  + 4G_2^2{{\tilde \vartheta }_4}\sum\limits_{i = 1}^n {\sum\limits_{t = 1}^T {{\mathbf{E}_\mathcal{C}}[ {{{\| {{z_{i,t + 1}} - {x_{i,t}}} \|}^2}} ]} } \\
\nonumber
&\;\; + 4G_2^2{{\tilde \vartheta }_4}\sum\limits_{i = 1}^n {\sum\limits_{t = 1}^T {{\mathbf{E}_\mathcal{C}}[ {{{\| {{{\hat z}_{i,t + 1}} - {z_{i,t + 1}}} \|}^2}} ]} } \\
\nonumber
& \le {\vartheta _3} + 2\sum\limits_{i = 1}^n {\sum\limits_{t = 1}^T {{\mathbf{E}_\mathcal{C}}[ {{{\| {{{[ {{g_{i,t}}( {{x_{i,t}}} )} ]}_ + }} \|}^2}} ]} }  + 4G_2^2{{\tilde \vartheta }_4}\sum\limits_{i = 1}^n {\sum\limits_{t = 1}^T {{\mathbf{E}_\mathcal{C}}[ {{{\| {{z_{i,t + 1}} - {x_{i,t}}} \|}^2}} ]} }  \\
&\;\; + 4n{\tilde p^2}CG_2^2{{\tilde \vartheta }_4}\sum\limits_{t = 1}^T {s_t^2},
\label{Lemma6-proof-ii-eq3}
\end{flalign}
where the second inequality holds due to \eqref{Lemma5-eq4}; and the last inequality holds since \eqref{Lemma4-proof-eq8} holds and $\{ {{s_t}} \}$ is nonincreasing.

From ${g_{i,t}}( y ) \le {0_{{m_i}}}$, $\forall i \in [ n ]$, $\forall t \in {\mathbb{N}_ + }$ when $y \in {\mathcal{X}_T}$, we have
\begin{flalign}
{\Lambda _T}( y ) \le 0.
\label{Lemma6-proof-ii-eq4}
\end{flalign}
Combining \eqref{Lemma6-proof-ii-eq3}--\eqref{Lemma6-proof-ii-eq4} and \eqref{Lemma5-eq2} yields
\begin{flalign}
\frac{1}{n}\sum\limits_{i = 1}^n {\sum\limits_{t = 1}^T {{\mathbf{E}_\mathcal{C}}[ {{{\| {{{[ {{g_t}( {{x_{i,t}}} )} ]}_ + }} \|}^2}} ]} }  \le {\vartheta _3} + {\vartheta _4}{{\tilde \Lambda }_T}( y ) + 4n{\tilde p^2}CG_2^2{{\tilde \vartheta }_4}\sum\limits_{t = 1}^T {s_t^2}, \forall y \in {\mathcal{X}_T}.
\label{Lemma6-proof-ii-eq5}
\end{flalign}
Using the H\"{o}lder’s inequality, we have
\begin{flalign}
{\Big( {\frac{1}{n}\sum\limits_{i = 1}^n {\sum\limits_{t = 1}^T {{\mathbf{E}_\mathcal{C}}[ {\| {{{[ {{g_t}( {{x_{i,t}}} )} ]}_ + }} \|} ]} } } \Big)^2} \le \frac{T}{n}\sum\limits_{i = 1}^n {\sum\limits_{t = 1}^T {{\mathbf{E}_\mathcal{C}}[ {{{\| {{{[ {{g_t}( {{x_{i,t}}} )} ]}_ + }} \|}^2}} ]} }.
\label{Lemma6-proof-ii-eq6}
\end{flalign}
Combining \eqref{Lemma6-proof-ii-eq5}--\eqref{Lemma6-proof-ii-eq6} yields \eqref{Lemma6-eq2}.

(iii)
We have
\begin{flalign}
\nonumber
\frac{1}{n}\sum\limits_{j = 1}^n {\sum\limits_{t = 1}^T {{\mathbf{E}_\mathcal{C}}[ {\| {{{[ {{g_t}( {{x_{j,t}}} )} ]}_ + }} \|} ]} } &\le \frac{1}{n}\sum\limits_{i = 1}^n {\sum\limits_{j = 1}^n {\sum\limits_{t = 1}^T {{\mathbf{E}_\mathcal{C}}[ {\| {{{[ {{g_{i,t}}( {{x_{j,t}}} )} ]}_ + }} \|} ]} } } \\
\nonumber
& = \frac{1}{n}\sum\limits_{i = 1}^n {\sum\limits_{j = 1}^n {\sum\limits_{t = 1}^T {{\mathbf{E}_\mathcal{C}}[ {\| {{{[ {{g_{i,t}}( {{x_{i,t}}} )} ]}_ + } + {{[ {{g_{i,t}}( {{x_{j,t}}} )} ]}_ + } - {{[ {{g_{i,t}}( {{x_{i,t}}} )} ]}_ + }} \|} ]} } } \\
\nonumber
& \le \frac{1}{n}\sum\limits_{i = 1}^n {\sum\limits_{j = 1}^n {\sum\limits_{t = 1}^T {{\mathbf{E}_\mathcal{C}}[ {\| {{{[ {{g_{i,t}}( {{x_{i,t}}} )} ]}_ + }} \| + \| {{g_{i,t}}( {{x_{i,t}}} ) - {g_{i,t}}( {{x_{j,t}}} )} \|} ]} } } \\
& \le \frac{1}{n}\sum\limits_{i = 1}^n {\sum\limits_{j = 1}^n {\sum\limits_{t = 1}^T {{\mathbf{E}_\mathcal{C}}[ {\| {{{[ {{g_{i,t}}( {{x_{i,t}}} )} ]}_ + }} \| + {G_2}\| {{x_{i,t}} - {x_{j,t}}} \|} ]} } },
\label{Lemma6-proof-iii-eq1}
\end{flalign}
where the first inequality holds due to ${g_t}( x ) = {\rm{col}}( {{g_{1,t}}( x ), \cdot  \cdot  \cdot ,{g_{n,t}}( x )} )$; the second inequality holds due to the nonexpansive property of the projection ${[  \cdot  ]_ + }$; and the last inequality holds due to \eqref{ass4-eq2}.

From \eqref{Lemma5-eq3}, we have
\begin{flalign}
\nonumber
&\;\;\;\;\;\frac{1}{n}\sum\limits_{i = 1}^n {\sum\limits_{j = 1}^n {\sum\limits_{t = 1}^T {{\mathbf{E}_\mathcal{C}}[ {\| {{x_{i,t}} - {x_{j,t}}} \|} ]} } } \\
\nonumber
& \le n{\vartheta _1} + {{\tilde \vartheta }_2}\sum\limits_{i = 1}^n {\sum\limits_{t = 1}^T {{\mathbf{E}_\mathcal{C}}[ {\| {\varepsilon _{i,t}^z} \|} ]} } \\
\nonumber
& \le n{\vartheta _1} + {{\tilde \vartheta }_2}\sum\limits_{i = 1}^n {\sum\limits_{t = 1}^T {{\mathbf{E}_\mathcal{C}}[ {\| {{z_{i,t + 1}} - {x_{i,t}}} \|} ]} }  + {{\tilde \vartheta }_2}\sum\limits_{i = 1}^n {\sum\limits_{t = 1}^T {{\mathbf{E}_\mathcal{C}}[ {\| {{{\hat z}_{i,t + 1}} - {z_{i,t + 1}}} \|} ]} } \\
& \le n{\vartheta _1} + {{{\tilde \vartheta }_2}}\sum\limits_{i = 1}^n {\sum\limits_{t = 1}^T {( {{G_1}{\alpha _t} + {G_2}{\gamma _0}{\mathbf{E}_\mathcal{C}}[ {\| {{{[ {{g_{i,t}}( {{x_{i,t}}} )} ]}_ + }} \|} ]} )} }  + n{\tilde p}\sqrt C {{\tilde \vartheta }_2}\sum\limits_{t = 1}^T {{s_t}},
\label{Lemma6-proof-iii-eq2}
\end{flalign}
where the last inequality holds since \eqref{Lemma5-eq5} and \eqref{Lemma4-proof-eq8} hold, and $\{ {{s_t}} \}$ is nonincreasing.

Combining \eqref{Lemma6-proof-iii-eq1}--\eqref{Lemma6-proof-iii-eq2} yields \eqref{Lemma6-eq3}.
\end{proof}

\hspace{-3mm}\emph{B. Proof of Theorem~1}

(i)
From \eqref{theorem1-eq1}, for any $T \in \mathbb{N}_ +$, we have
\begin{flalign}
\sum\limits_{t = 1}^T {{\alpha _t}}  = {\alpha _0}\sum\limits_{t = 1}^T {\frac{1}{{{t^{{\theta _1}}}}}}  = {\alpha _0}\Big(\sum\limits_{t = 2}^T {\frac{1}{{{t^{{\theta _1}}}}}}  + 1\Big) \le {\alpha _0}\Big(\int_1^T {\frac{1}{{{t^{{\theta _1}}}}}dt}  + 1\Big) \le \frac{{{{\alpha _0}T^{1 - {\theta _1}}}}}{{1 - {\theta _1}}}.
\label{Theorem1-proof-i-eq1}
\end{flalign}
Similar to the way to get \eqref{Theorem1-proof-i-eq1}, from \eqref{theorem1-eq1} with ${\theta _2} \in ( {{\theta _1},1} )$, for any $T \in \mathbb{N}_ +$, we have
\begin{flalign}
&\sum\limits_{t = 1}^T {{s_t}}  = {s_0}\sum\limits_{t = 1}^T {\frac{1}{{{t^{{\theta _2}}}}}}  \le \frac{{{{s_0}T^{1 - {\theta _2}}}}}{{1 - {\theta _2}}}, \label{Theorem1-proof-i-eq2}\\
&\sum\limits_{t = 1}^T {\frac{{{s_t}}}{{{\alpha _t}}}}  = \frac{{{s_0}}}{{{\alpha _0}}}\sum\limits_{t = 1}^T {\frac{1}{{{t^{{\theta _2} - {\theta _1}}}}}}  \le \frac{{{{s_0}T^{1 + {\theta _1} - {\theta _2}}}}}{({1 + {\theta _1} - {\theta _2}}){\alpha _0}}.
\label{Theorem1-proof-i-eq3}
\end{flalign}
We have
\begin{flalign}
\nonumber
\frac{1}{n}\sum\limits_{i = 1}^n {\sum\limits_{t = 1}^T {{\mathbf{E}_\mathcal{C}}[ {{\Delta _{i,t}}( y )} ]} }  &= \frac{1}{n}\sum\limits_{i = 1}^n {\sum\limits_{t = 1}^T {{\mathbf{E}_\mathcal{C}}\Big[ {\frac{1}{{2{\alpha _t}}}( {{{\| {y - {x_{i,t}}} \|}^2} - {{\| {y - {x_{i,t + 1}}} \|}^2}} )} \Big]} }, \\
\nonumber
& = \frac{1}{n}\sum\limits_{i = 1}^n {\sum\limits_{t = 1}^T {\frac{1}{2}{\mathbf{E}_\mathcal{C}}\Big[ {\frac{{{{\| {y - {x_{i,t}}} \|}^2}}}{{{\alpha _{t - 1}}}} - \frac{{{{\| {y - {x_{i,t + 1}}} \|}^2}}}{{{\alpha _t}}} + ( {\frac{1}{{{\alpha _t}}} - \frac{1}{{{\alpha _{t - 1}}}}} ){{\| {y - {x_{i,t}}} \|}^2}} \Big]} } \\
\nonumber
& \le \frac{1}{n}\sum\limits_{i = 1}^n {\frac{1}{2}{\mathbf{E}_\mathcal{C}}\Big[ {\frac{{{{\| {y - {x_{i,1}}} \|}^2}}}{{{\alpha _0}}} - \frac{{{{\| {y - {x_{i,T + 1}}} \|}^2}}}{{{\alpha _T}}} + 4R{{( \mathbb{X} )}^2}( {\frac{1}{{{\alpha _T}}} - \frac{1}{{{\alpha _0}}}} )} \Big]} \\
&= \frac{{2R{{( \mathbb{X} )}^2}}}{{{\alpha _T}}} \le \frac{{2R{{( \mathbb{X} )}^2}}}{{{\alpha _0}}}{T^{{\theta _1}}}, \forall {y} \in \mathbb{X},
\label{Theorem1-proof-i-eq4}
\end{flalign}
where the first inequality holds since \eqref{ass1-eq1} holds and $\{ {{\alpha _t}} \}$ is nonincreasing; the last equality holds due to \eqref{ass1-eq1}; and the last inequality holds due to \eqref{theorem1-eq1}.

Combining \eqref{Lemma6-eq1} and \eqref{Theorem1-proof-i-eq1}--\eqref{Theorem1-proof-i-eq4}, from the arbitrariness of ${y} \in {\mathcal{X}_T}$, we have
\begin{flalign}
\nonumber
\mathbf{E}_\mathcal{C}[{{\rm{Net}\mbox{-}\rm{Reg}}( T )}] &\le {\hat \vartheta } + \frac{{{\vartheta _2}{\alpha _0}}}{{1 - {\theta _1}}}{T^{1 - {\theta _1}}} + \frac{{\tilde p\sqrt C {{\tilde \vartheta }_1}{s_0}}}{{1 - {\theta _2}}}{T^{1 - {\theta _2}}} \\
&\;\; + \frac{{2{\tilde p}\sqrt C R( \mathbb{X} ){s_0}}}{({1 + {\theta _1} - {\theta _2}}){\alpha _0}}{T^{1 + {\theta _1} - {\theta _2}}} + \frac{{2R{{( \mathbb{X} )}^2}}}{{{\alpha _0}}}{T^{{\theta _1}}}.
\label{Theorem1-proof-i-eq5}
\end{flalign}
From \eqref{theorem1-eq1} with ${\theta _2} = 1$, for any $T \in \mathbb{N}_ +$, we have
\begin{flalign}
&\sum\limits_{t = 1}^T {{s_t}}  = {s_0}\sum\limits_{t = 1}^T {\frac{1}{t}}  \le {s_0}\Big( {\int_1^T {\frac{1}{t}dt}  + 1} \Big) \le {s_0}( {\log ( T ) + 1} ) \le 2{s_0}\log ( T ), {\rm{if }}\; T \ge 3 \label{Theorem1-proof-i-eq6}\\
&\sum\limits_{t = 1}^T {\frac{{{s_t}}}{{{\alpha _t}}}}  = \frac{{{s_0}}}{{{\alpha _0}}}\sum\limits_{t = 1}^T {\frac{1}{{{t^{1 - {\theta _1}}}}}}  \le \frac{{{{s_0}T^{ {\theta _1}}}}}{{{\theta _1} }{\alpha _0}}.
\label{Theorem1-proof-i-eq7}
\end{flalign}
Combining \eqref{Lemma6-eq1}, \eqref{Theorem1-proof-i-eq1}, \eqref{Theorem1-proof-i-eq4} and \eqref{Theorem1-proof-i-eq6}--\eqref{Theorem1-proof-i-eq7}, from the arbitrariness of ${y} \in {\mathcal{X}_T}$, we have
\begin{flalign}
\nonumber
\mathbf{E}_\mathcal{C}[{{\rm{Net}\mbox{-}\rm{Reg}}( T )}] &\le {\hat \vartheta } + \frac{{{\vartheta _2}{\alpha _0}}}{{1 - {\theta _1}}}{T^{1 - {\theta _1}}} + 2\tilde p\sqrt C {{\tilde \vartheta }_1}{s_0}\log ( T ) \\
&\;\; + \frac{{2{\tilde p}\sqrt C R( \mathbb{X} ){s_0}}}{{ {\theta _1}}{\alpha _0}}{T^{{\theta _1}}} + \frac{{2R{{( \mathbb{X} )}^2}}}{{{\alpha _0}}}{T^{{\theta _1}}}.
\label{Theorem1-proof-i-eq8}
\end{flalign}
From \eqref{theorem1-eq1} with ${\theta _2} \in ( {1,1 + {\theta _1}} )$, for any $T \in \mathbb{N}_ +$, there exists a constant ${\rm Z_1} > 0$ such that
\begin{flalign}
&\sum\limits_{t = 1}^T {{s_t}}  = {s_0}\sum\limits_{t = 1}^T {\frac{1}{{{t^{{\theta _2}}}}}}  \le {\rm Z_1}{s_0}, \label{Theorem1-proof-i-eq9}\\
&\sum\limits_{t = 1}^T {\frac{{{s_t}}}{{{\alpha _t}}}}  = \frac{{{s_0}}}{{{\alpha _0}}}\sum\limits_{t = 1}^T {\frac{1}{{{t^{{\theta _2} - {\theta _1}}}}}}  \le \frac{{{{s_0}T^{1 + {\theta _1} - {\theta _2}}}}}{({1 + {\theta _1} - {\theta _2}}){\alpha _0}}.
\label{Theorem1-proof-i-eq10}
\end{flalign}
Combining \eqref{Lemma6-eq1}, \eqref{Theorem1-proof-i-eq1}, \eqref{Theorem1-proof-i-eq4} and \eqref{Theorem1-proof-i-eq9}--\eqref{Theorem1-proof-i-eq10}, from the arbitrariness of ${y} \in {\mathcal{X}_T}$, we have
\begin{flalign}
\nonumber
\mathbf{E}_\mathcal{C}[{{\rm{Net}\mbox{-}\rm{Reg}}( T )}] &\le {\hat \vartheta } + \frac{{{\vartheta _2}{\alpha _0}}}{{1 - {\theta _1}}}{T^{1 - {\theta _1}}} + \tilde p\sqrt C {{\tilde \vartheta }_1}{\rm Z_1}{s_0} \\
&\;\; + \frac{{2{\tilde p}\sqrt C R( \mathbb{X} ){s_0}}}{({1 + {\theta _1} - {\theta _2}}){\alpha _0}}{T^{1 + {\theta _1} - {\theta _2}}} + \frac{{2R{{( \mathbb{X} )}^2}}}{{{\alpha _0}}}{T^{{\theta _1}}}.
\label{Theorem1-proof-i-eq11}
\end{flalign}
From \eqref{theorem1-eq1} with ${\theta _2} = 1 + {\theta _1}$, for any $T \in \mathbb{N}_ +$, there exists a constant ${\rm Z_2} > 0$ such that
\begin{flalign}
&\sum\limits_{t = 1}^T {{s_t}}  = {s_0}\sum\limits_{t = 1}^T {\frac{1}{{{t^{{\theta _2}}}}}}  \le {\rm Z_2}{s_0}, \label{Theorem1-proof-i-eq12}\\
&\sum\limits_{t = 1}^T {\frac{{{s_t}}}{{{\alpha _t}}}}  = \frac{{{s_0}}}{{{\alpha _0}}}\sum\limits_{t = 1}^T {\frac{1}{t}}  \le \frac{{2{s_0}}}{{{\alpha _0}}}\log ( T ), {\rm{if }}\;T \ge 3.
\label{Theorem1-proof-i-eq13}
\end{flalign}
Combining \eqref{Lemma6-eq1}, \eqref{Theorem1-proof-i-eq1}, \eqref{Theorem1-proof-i-eq4} and \eqref{Theorem1-proof-i-eq12}--\eqref{Theorem1-proof-i-eq13}, from the arbitrariness of ${y} \in {\mathcal{X}_T}$, we have
\begin{flalign}
\nonumber
\mathbf{E}_\mathcal{C}[{{\rm{Net}\mbox{-}\rm{Reg}}( T )}] &\le {\hat \vartheta } + \frac{{{\vartheta _2}{\alpha _0}}}{{1 - {\theta _1}}}{T^{1 - {\theta _1}}} + \tilde p\sqrt C {{\tilde \vartheta }_1}{\rm Z_2}{s_0} \\
&\;\; + \frac{{4{\tilde p}\sqrt C R( \mathbb{X} ){s_0}}}{{{\alpha _0}}}\log ( T ) + \frac{{2R{{( \mathbb{X} )}^2}}}{{{\alpha _0}}}{T^{{\theta _1}}}.
\label{Theorem1-proof-i-eq14}
\end{flalign}
From \eqref{theorem1-eq1} with ${\theta _2} > 1 + {\theta _1}$, for any $T \in \mathbb{N}_ +$, there exists a constant ${\rm Z_3} > 0$ such that
\begin{flalign}
&\sum\limits_{t = 1}^T {{s_t}}  = {s_0}\sum\limits_{t = 1}^T {\frac{1}{{{t^{{\theta _2}}}}}}  \le {\rm Z_3}{s_0}, \label{Theorem1-proof-i-eq15}\\
&\sum\limits_{t = 1}^T {\frac{{{s_t}}}{{{\alpha _t}}}} = \frac{{{s_0}}}{{{\alpha _0}}}\sum\limits_{t = 1}^T {\frac{1}{{{t^{{\theta _2} - {\theta _1}}}}}}  \le {\rm Z_3}\frac{{{s_0}}}{{{\alpha _0}}}.
\label{Theorem1-proof-i-eq16}
\end{flalign}
Combining \eqref{Lemma6-eq1}, \eqref{Theorem1-proof-i-eq1}, \eqref{Theorem1-proof-i-eq4} and \eqref{Theorem1-proof-i-eq15}--\eqref{Theorem1-proof-i-eq16}, from the arbitrariness of ${y} \in {\mathcal{X}_T}$, we have
\begin{flalign}
\nonumber
\mathbf{E}_\mathcal{C}[{{\rm{Net}\mbox{-}\rm{Reg}}( T )}] &\le {\hat \vartheta } + \frac{{{\vartheta _2}{\alpha _0}}}{{1 - {\theta _1}}}{T^{1 - {\theta _1}}} + \tilde p\sqrt C {{\tilde \vartheta }_1}{\rm Z_3}{s_0} \\
&\;\; + \frac{{2{\tilde p}\sqrt C R( \mathbb{X} ){\rm Z_3}{s_0}}}{{{\alpha _0}}} + \frac{{2R{{( \mathbb{X} )}^2}}}{{{\alpha _0}}}{T^{{\theta _1}}}.
\label{Theorem1-proof-i-eq17}
\end{flalign}

From \eqref{Theorem1-proof-i-eq5}, \eqref{Theorem1-proof-i-eq8}, \eqref{Theorem1-proof-i-eq11}, \eqref{Theorem1-proof-i-eq14}, and \eqref{Theorem1-proof-i-eq17}, we know that \eqref{theorem1-eq2} holds.

(ii)
From \eqref{theorem1-eq1}, for any $T \in \mathbb{N}_ +$, we have
\begin{flalign}
&\sum\limits_{t = 1}^T {\frac{1}{{{\gamma _t}}} = } \sum\limits_{t = 1}^T {\frac{{{\alpha _t}}}{{{\gamma _0}}}}  \le \frac{{{\alpha _0}{T^{1 - {\theta _1}}}}}{{( {1 - {\theta _1}} ){\gamma _0}}}, \label{Theorem1-proof-ii-eq1}\\
&\sum\limits_{t = 1}^T \frac{{{\gamma _0}}}{{\gamma _t^2}} \le  {\alpha _0}\sum\limits_{t = 1}^T {\frac{1}{{{\gamma _t}}}}  = \frac{{{\alpha _0^2}}}{{{\gamma _0}}}\sum\limits_{t = 1}^T {\frac{1}{{{t^{{\theta _1}}}}}}  \le \frac{{{\alpha _0^2}{T^{1 - {\theta _1}}}}}{{( {1 - {\theta _1}} ){\gamma _0}}}, \label{Theorem1-proof-ii-eq2}\\
&\sum\limits_{t = 1}^T {s_t^2}  \le {s _0}\sum\limits_{t = 1}^T {{s_t}}.
\label{Theorem1-proof-ii-eq3}
\end{flalign}
From $\eqref{ass1-eq1}$, we have
\begin{flalign}
\sum\limits_{i = 1}^n {\frac{{{{\| {y - {x_{i,1}}} \|}^2}}}{{{\gamma _0}}}}  \le \frac{{4nR{{( \mathbb{X} )}^2}}}{{{\gamma _0}}}, \forall {y} \in \mathbb{X}.
\label{Theorem1-proof-ii-eq4}
\end{flalign}
Combining \eqref{Lemma6-eq2} and \eqref{Theorem1-proof-ii-eq1}--\eqref{Theorem1-proof-ii-eq4}, from \eqref{theorem1-eq1} with ${\theta _2} \in ( {{\theta _1},1} )$ and \eqref{Theorem1-proof-i-eq2}, we have
\begin{flalign}
\nonumber
{( {\frac{1}{n}\sum\limits_{i = 1}^n {\sum\limits_{t = 1}^T {{\mathbf{E}_\mathcal{C}}[ {\| {{{[ {{g_t}( {{x_{i,t}}} )} ]}_ + }} \|} ]} } } )^2} &\le {\vartheta _3}T + \frac{{2nR{{( \mathbb{X} )}^2}{\vartheta _4}}}{{{\gamma _0}}}T + \frac{{2n{G_1}R( \mathbb{X} ){\vartheta _4}{\alpha _0}}}{{( {1 - {\theta _1}} ){\gamma _0}}}{T^{2 - {\theta _1}}} + \frac{{2nG_1^2{\vartheta _4}{\alpha _0^2}}}{{( {1 - {\theta _1}} ){\gamma _0}}}{T^{2 - {\theta _1}}}\\
&\;\; + \frac{{2n{\tilde p}\sqrt C R( \mathbb{X} ){\vartheta _4}{s_0}}}{{( {1 - {\theta _2}} ){\gamma _0}}}{T^{2 - {\theta _2}}} + \frac{{4n{\tilde p^2}CG_2^2{{\tilde \vartheta }_4}{s_0^2}}}{{( {1 - {\theta _2}} )}}{T^{2 - {\theta _2}}}.
\label{Theorem1-proof-ii-eq5}
\end{flalign}
Combining \eqref{Lemma6-eq2} and \eqref{Theorem1-proof-ii-eq1}--\eqref{Theorem1-proof-ii-eq4}, from \eqref{theorem1-eq1} with ${\theta _2} = 1$ and \eqref{Theorem1-proof-i-eq6}, we have
\begin{flalign}
\nonumber
{( {\frac{1}{n}\sum\limits_{i = 1}^n {\sum\limits_{t = 1}^T {{\mathbf{E}_\mathcal{C}}[ {\| {{{[ {{g_t}( {{x_{i,t}}} )} ]}_ + }} \|} ]} } } )^2} &\le {\vartheta _3}T + \frac{{2nR{{( \mathbb{X} )}^2}{\vartheta _4}}}{{{\gamma _0}}}T + \frac{{2n{G_1}R( \mathbb{X} ){\vartheta _4}{\alpha _0}}}{{( {1 - {\theta _1}} ){\gamma _0}}}{T^{2 - {\theta _1}}} + \frac{{2nG_1^2{\vartheta _4}{\alpha _0^2}}}{{( {1 - {\theta _1}} ){\gamma _0}}}{T^{2 - {\theta _1}}}\\
&\;\; + \frac{{4n{\tilde p}\sqrt C R( \mathbb{X} ){\vartheta _4}{s_0}}}{{{\gamma _0}}}T\log ( T ) + 8n{\tilde p^2}CG_2^2{{\tilde \vartheta }_4}{s_0^2}T\log ( T ).
\label{Theorem1-proof-ii-eq6}
\end{flalign}
Combining \eqref{Lemma6-eq2} and \eqref{Theorem1-proof-ii-eq1}--\eqref{Theorem1-proof-ii-eq4}, from \eqref{theorem1-eq1} with ${\theta _2} > 1$, \eqref{Theorem1-proof-i-eq9}, \eqref{Theorem1-proof-i-eq12}, and \eqref{Theorem1-proof-i-eq15}, choosing ${\rm Z} = \max \{ {{{\rm Z}_1},{{\rm Z}_2},{{\rm Z}_3}} \}$, we have
\begin{flalign}
\nonumber
{( {\frac{1}{n}\sum\limits_{i = 1}^n {\sum\limits_{t = 1}^T {{\mathbf{E}_\mathcal{C}}[ {\| {{{[ {{g_t}( {{x_{i,t}}} )} ]}_ + }} \|} ]} } } )^2} &\le {\vartheta _3}T + \frac{{2nR{{( \mathbb{X} )}^2}{\vartheta _4}}}{{{\gamma _0}}}T + \frac{{2n{G_1}R( \mathbb{X} ){\vartheta _4}{\alpha _0}}}{{( {1 - {\theta _1}} ){\gamma _0}}}{T^{2 - {\theta _1}}} + \frac{{2nG_1^2{\vartheta _4}{\alpha _0^2}}}{{( {1 - {\theta _1}} ){\gamma _0}}}{T^{2 - {\theta _1}}}\\
&\;\;  + \frac{{2n{\tilde p}\sqrt C R( \mathbb{X} ){\vartheta _4}{\rm Z}{s_0}}}{{{\gamma _0}}}T + 4n{\tilde p^2}CG_2^2{{\tilde \vartheta }_4}{\rm Z}{s_0^2}T.
\label{Theorem1-proof-ii-eq7}
\end{flalign}
From \eqref{Theorem1-proof-ii-eq5}--\eqref{Theorem1-proof-ii-eq7}, we know that \eqref{theorem1-eq3} holds.

(iii)
We have
\begin{flalign}
\nonumber
{\mathbf{E}_\mathcal{C}}[{\Lambda _T}( {{x_s}} )] &= {\mathbf{E}_\mathcal{C}}\Big[\sum\limits_{i = 1}^n {\sum\limits_{t = 1}^T {\frac{{v_{i,t + 1}^T{g_{i,t}}( {{x_s}} )}}{{{\gamma _t}}}} }\Big]  = {\mathbf{E}_\mathcal{C}}\Big[\sum\limits_{i = 1}^n {\sum\limits_{t = 1}^T {[ {{g_{i,t}}( {{x_{i,t}}} )} ]_ + ^T{g_{i,t}}( {{x_s}} )} }\Big]  \\
\nonumber
& \le  - {\mathbf{E}_\mathcal{C}}\Big[\sum\limits_{i = 1}^n {\sum\limits_{t = 1}^T {{\varsigma _s}[ {{g_{i,t}}( {{x_{i,t}}} )} ]_ + ^T{\mathbf{1}_{{m_i}}}} }\Big] =  - {\varsigma _s}\sum\limits_{i = 1}^n {\sum\limits_{t = 1}^T {\mathbf{E}_\mathcal{C}}[{{{\| {{{[ {{g_{i,t}}( {{x_{i,t}}} )} ]}_ + }} \|}_1}}] } \\
& \le  - {\varsigma _s}\sum\limits_{i = 1}^n {\sum\limits_{t = 1}^T {\mathbf{E}_\mathcal{C}}[{\| {{{[ {{g_{i,t}}( {{x_{i,t}}} )} ]}_ + }} \|}] },
\label{Theorem1-proof-iii-eq1}
\end{flalign}
where the second equality holds due to \eqref{Algorithm1-eq5}; and the first inequality holds due to Assumption~6.

Selecting $y = {x_s}$ in \eqref{Lemma5-eq2}, from \eqref{Algorithm1-eq5} and \eqref{Theorem1-proof-iii-eq1}, we have
\begin{flalign}
{\varsigma _s}\sum\limits_{i = 1}^n {\sum\limits_{t = 1}^T {\mathbf{E}_\mathcal{C}}[{\| {{{[ {{g_{i,t}}( {{x_{i,t}}} )} ]}_ + }} \|}] }  \le {{\tilde \Lambda }_T}( {{x_s}} ),
\label{Theorem1-proof-iii-eq2}
\end{flalign}
Combining \eqref{Theorem1-proof-iii-eq2} and \eqref{Theorem1-proof-ii-eq1}--\eqref{Theorem1-proof-ii-eq4}, from \eqref{theorem1-eq1} with ${\theta _2} \in ( {{\theta _1},1} )$ and \eqref{Theorem1-proof-i-eq2}, we have
\begin{flalign}
\nonumber
\sum\limits_{i = 1}^n {\sum\limits_{t = 1}^T {{\mathbf{E}_\mathcal{C}}[ {\| {{{[ {{g_{i,t}}( {{x_{i,t}}} )} ]}_ + }} \|} ]} } &\le \frac{{2nR{{( \mathbb{X} )}^2}}}{{{\varsigma _s}{\gamma _0}}} + \frac{{2n{G_1}R( \mathbb{X} ){\alpha _0}}}{{( {1 - {\theta _1}} ){\varsigma _s}{\gamma _0}}}{T^{1 - {\theta _1}}} + \frac{{2nG_1^2{\alpha _0^2}}}{{( {1 - {\theta _1}} ){\varsigma _s}{\gamma _0}}}{T^{1 - {\theta _1}}} \\
&\;\; + \frac{{2n{\tilde p}\sqrt C R( \mathbb{X} ){s_0}}}{{( {1 - {\theta _2}} ){\varsigma _s}{\gamma _0}}}{T^{1 - {\theta _2}}}.
\label{Theorem1-proof-iii-eq3}
\end{flalign}
Combining \eqref{Lemma6-eq3}, \eqref{Theorem1-proof-i-eq1}, \eqref{Theorem1-proof-iii-eq3}, from \eqref{theorem1-eq1} with ${\theta _2} \in ( {{\theta _1},1} )$ and \eqref{Theorem1-proof-i-eq2}, we have
\begin{flalign}
\nonumber
\mathbf{E}_\mathcal{C}[{{\rm{Net} \mbox{-} \rm{CCV}}( T )}] & \le n{G_2}{\vartheta _1} + \frac{{2nR{{( \mathbb{X} )}^2}{\vartheta _6}}}{{{\varsigma _s}{\gamma _0}}}  + \frac{{2n{G_1}R( \mathbb{X} ){\vartheta _6}{\alpha _0}}}{{( {1 - {\theta _1}} ){\varsigma _s}{\gamma _0}}}{T^{1 - {\theta _1}}} + \frac{{2nG_1^2{\vartheta _6}{\alpha _0^2}}}{{( {1 - {\theta _1}} ){\varsigma _s}{\gamma _0}}}{T^{1 - {\theta _1}}} \\
&\;\; + \frac{{{\vartheta _5}}}{{1 - {\theta _1}}}{T^{1 - {\theta _1}}} + \frac{{2n{\tilde p}\sqrt C R( \mathbb{X} ){\vartheta _6}{s_0}}}{{( {1 - {\theta _2}} ){\varsigma _s}{\gamma _0}}}{T^{1 - {\theta _2}}} + \frac{{n{\tilde p}\sqrt C {G_2}{{\tilde \vartheta }_2}{s_0}}}{{( {1 - {\theta _2}} )}}{T^{1 - {\theta _2}}}.
\label{Theorem1-proof-iii-eq4}
\end{flalign}
Combining \eqref{Theorem1-proof-iii-eq2} and \eqref{Theorem1-proof-ii-eq1}--\eqref{Theorem1-proof-ii-eq4}, from \eqref{theorem1-eq1} with ${\theta _2} = 1$ and \eqref{Theorem1-proof-i-eq6}, we have
\begin{flalign}
\nonumber
\sum\limits_{i = 1}^n {\sum\limits_{t = 1}^T {{\mathbf{E}_\mathcal{C}}[ {\| {{{[ {{g_{i,t}}( {{x_{i,t}}} )} ]}_ + }} \|} ]} } &\le \frac{{2nR{{( \mathbb{X} )}^2}}}{{{\varsigma _s}{\gamma _0}}} + \frac{{2n{G_1}R( \mathbb{X} ){\alpha _0}}}{{( {1 - {\theta _1}} ){\varsigma _s}{\gamma _0}}}{T^{1 - {\theta _1}}} + \frac{{2nG_1^2{\alpha _0^2}}}{{( {1 - {\theta _1}} ){\varsigma _s}{\gamma _0}}}{T^{1 - {\theta _1}}} \\
&\;\; + \frac{{4n{\tilde p}\sqrt C R( \mathbb{X} ){s_0}}}{{{\varsigma _s}{\gamma _0}}}\log ( T ).
\label{Theorem1-proof-iii-eq5}
\end{flalign}
Combining \eqref{Lemma6-eq3}, \eqref{Theorem1-proof-i-eq1}, \eqref{Theorem1-proof-iii-eq5}, from \eqref{theorem1-eq1} with ${\theta _2} = 1$ and \eqref{Theorem1-proof-i-eq6}, we have
\begin{flalign}
\nonumber
\mathbf{E}_\mathcal{C}[{{\rm{Net} \mbox{-} \rm{CCV}}( T )}] & \le n{G_2}{\vartheta _1} + \frac{{2nR{{( \mathbb{X} )}^2}{\vartheta _6}}}{{{\varsigma _s}{\gamma _0}}} + \frac{{2n{G_1}R( \mathbb{X} ){\vartheta _6}{\alpha _0}}}{{( {1 - {\theta _1}} ){\varsigma _s}{\gamma _0}}}{T^{1 - {\theta _1}}} + \frac{{2nG_1^2{\vartheta _6}{\alpha _0^2}}}{{( {1 - {\theta _1}} ){\varsigma _s}{\gamma _0}}}{T^{1 - {\theta _1}}} \\
&\;\; + \frac{{{\vartheta _5}}}{{1 - {\theta _1}}}{T^{1 - {\theta _1}}} + \frac{{4n{\tilde p}\sqrt C R( \mathbb{X} ){\vartheta _6}{s_0}}}{{{\varsigma _s}{\gamma _0}}}\log ( T ) + 2n{\tilde p}\sqrt C {G_2}{{\tilde \vartheta }_2}{s_0}\log ( T ).
\label{Theorem1-proof-iii-eq6}
\end{flalign}
Combining \eqref{Theorem1-proof-iii-eq2} and \eqref{Theorem1-proof-ii-eq1}--\eqref{Theorem1-proof-ii-eq4}, from \eqref{theorem1-eq1} with ${\theta _2} > 1$, \eqref{Theorem1-proof-i-eq9}, \eqref{Theorem1-proof-i-eq12}, and \eqref{Theorem1-proof-i-eq15}, we have
\begin{flalign}
\nonumber
\sum\limits_{i = 1}^n {\sum\limits_{t = 1}^T {{\mathbf{E}_\mathcal{C}}[ {\| {{{[ {{g_{i,t}}( {{x_{i,t}}} )} ]}_ + }} \|} ]} } &\le \frac{{2nR{{( \mathbb{X} )}^2}}}{{{\varsigma _s}{\gamma _0}}} + \frac{{2n{G_1}R( \mathbb{X} ){\alpha _0}}}{{( {1 - {\theta _1}} ){\varsigma _s}{\gamma _0}}}{T^{1 - {\theta _1}}} + \frac{{2nG_1^2{\alpha _0^2}}}{{( {1 - {\theta _1}} ){\varsigma _s}{\gamma _0}}}{T^{1 - {\theta _1}}} \\
&\;\; + \frac{{2n{\tilde p}\sqrt C R( \mathbb{X} ){\rm Z}{s_0}}}{{{\varsigma _s}{\gamma _0}}}.
\label{Theorem1-proof-iii-eq7}
\end{flalign}
Combining \eqref{Lemma6-eq3}, \eqref{Theorem1-proof-i-eq1}, \eqref{Theorem1-proof-iii-eq7}, from \eqref{theorem1-eq1} with ${\theta _2} > 1$, \eqref{Theorem1-proof-i-eq9}, \eqref{Theorem1-proof-i-eq12}, and \eqref{Theorem1-proof-i-eq15}, we have
\begin{flalign}
\nonumber
\mathbf{E}_\mathcal{C}[{{\rm{Net} \mbox{-} \rm{CCV}}( T )}] & \le n{G_2}{\vartheta _1} + \frac{{2nR{{( \mathbb{X} )}^2}{\vartheta _6}}}{{{\varsigma _s}{\gamma _0}}} + \frac{{2n{G_1}R( \mathbb{X} ){\vartheta _6}{\alpha _0}}}{{( {1 - {\theta _1}} ){\varsigma _s}{\gamma _0}}}{T^{1 - {\theta _1}}} + \frac{{2nG_1^2{\vartheta _6}{\alpha _0^2}}}{{( {1 - {\theta _1}} ){\varsigma _s}{\gamma _0}}}{T^{1 - {\theta _1}}} \\
&\;\; + \frac{{{\vartheta _5}}}{{1 - {\theta _1}}}{T^{1 - {\theta _1}}} + \frac{{2n{\tilde p}\sqrt C R( \mathbb{X} ){\vartheta _6}{\rm Z}{s_0}}}{{{\varsigma _s}{\gamma _0}}} + n{\tilde p}\sqrt C {G_2}{{\tilde \vartheta }_2}{\rm Z}{s_0}.
\label{Theorem1-proof-iii-eq8}
\end{flalign}
From \eqref{Theorem1-proof-iii-eq4}, \eqref{Theorem1-proof-iii-eq6}, and \eqref{Theorem1-proof-iii-eq8}, we know that \eqref{theorem1-eq4} holds.

\hspace{-3mm}\emph{C. Proof of Theorem~2}

(i)
From \eqref{theorem2-eq1}, for any $T \in \mathbb{N}_ +$, we have
\begin{flalign}
&{\alpha _T} = {\alpha _0}\sqrt {\frac{{{\Psi _T}}}{T}}  \le \frac{{{\alpha _0}\mu }}{{1 - \mu }}, \label{Theorem2-proof-i-eq1} \\
&\sum\limits_{t = 1}^T {{\alpha _t}}  = {\alpha _0}\sum\limits_{t = 1}^T {\sqrt {\frac{{{\Psi _t}}}{t}}  \le {\alpha _0}\sqrt {{\Psi _T}} } \sum\limits_{t = 1}^T {\frac{1}{{\sqrt t }}}  \le 2{\alpha _0}\sqrt {{\Psi _T}T}  \le \frac{{2{\alpha _0}\mu }}{{1 - \mu }}\sqrt T, \label{Theorem2-proof-i-eq2}\\
&\sum\limits_{t = 1}^T {{s_t}}  = {s_0}\sum\limits_{t = 1}^T {{\mu ^t}}  = \frac{{{s_0}\mu ( {1 - {\mu ^T}} )}}{{1 - \mu }} \le \frac{{{s_0}\mu }}{{1 - \mu }}, \label{Theorem2-proof-i-eq3}\\
&\sum\limits_{t = 1}^T {\frac{{{s_t}}}{{{\alpha _t}}}}  = \frac{{{s_0}}}{{{\alpha _0}}}\sum\limits_{t = 1}^T {\frac{{{\mu ^t}\sqrt t }}{{\sqrt {{\Psi _t}} }}}  \le \frac{{{s_0}}}{{{\alpha _0}}}\sum\limits_{t = 1}^T {\sqrt {{\mu ^t}} }  \le \frac{{{s_0}\sqrt \mu  }}{{{\alpha _0}( {1 - \sqrt \mu  } )}}.
\label{Theorem2-proof-i-eq4}
\end{flalign}
Combining \eqref{Lemma6-eq1}, \eqref{Theorem1-proof-i-eq4}, \eqref{Theorem2-proof-i-eq1}--\eqref{Theorem2-proof-i-eq4}, from the arbitrariness of ${y} \in {\mathcal{X}_T}$, we have
\begin{flalign}
\nonumber
\mathbf{E}_\mathcal{C}[{{\rm{Net}\mbox{-}\rm{Reg}}( T )}] &\le {\hat \vartheta } + \frac{{2{\vartheta _2}{\alpha _0}\mu }}{{1 - \mu }}\sqrt T  + \frac{{\tilde p\sqrt C {{\tilde \vartheta }_1}{s_0}\mu }}{{1 - \mu }} \\
&\;\; + \frac{{2\tilde p\sqrt C R( \mathbb{X} ){s_0}\sqrt \mu  }}{{{\alpha _0}( {1 - \sqrt \mu  } )}} + \frac{{2R{{( \mathbb{X} )}^2}( {1 - \mu } )}}{{{\alpha _0}\mu }}.
\label{Theorem2-proof-i-eq5}
\end{flalign}
From \eqref{Theorem2-proof-i-eq5}, we know that \eqref{theorem2-eq2} holds.

(ii)
From \eqref{theorem1-eq1}, for any $T \in \mathbb{N}_ +$, we have
\begin{flalign}
&\sum\limits_{t = 1}^T {\frac{1}{{{\gamma _t}}}}  = \frac{1}{{{\gamma _0}}}\sum\limits_{t = 1}^T {{\alpha _t} \le } \frac{{2{\alpha _0}\mu }}{{{\gamma _0}( {1 - \mu } )}}\sqrt T, \label{Theorem2-proof-ii-eq1}\\
&\sum\limits_{t = 1}^T {\frac{{{\gamma _0}}}{{\gamma _t^2}}}  = \frac{{\alpha _0^2}}{{{\gamma _0}}}\sum\limits_{t = 1}^T {\frac{{{\Psi _t}}}{t}}  \le \frac{{\alpha _0^2{\Psi _T}}}{{{\gamma _0}}}\sum\limits_{t = 1}^T {\frac{1}{t}}  \le \frac{{2\alpha _0^2\mu }}{{{\gamma _0}( {1 - \mu } )}}\log ( T ), {\rm{if }}\; T \ge 3, \label{Theorem2-proof-ii-eq2}\\
&\sum\limits_{t = 1}^T {s_t^2}  = s_0^2\sum\limits_{t = 1}^T {{\mu ^{2t}}}  \le \frac{{s_0^2{\mu ^2}}}{{1 - {\mu ^2}}}.
\label{Theorem2-proof-ii-eq3}
\end{flalign}
Combining \eqref{Lemma6-eq2}, \eqref{Theorem1-proof-ii-eq4}, \eqref{Theorem2-proof-i-eq3}, and \eqref{Theorem2-proof-ii-eq1}--\eqref{Theorem2-proof-ii-eq3}, we have
\begin{flalign}
\nonumber
{( {\frac{1}{n}\sum\limits_{i = 1}^n {\sum\limits_{t = 1}^T {{\mathbf{E}_\mathcal{C}}[ {\| {{{[ {{g_t}( {{x_{i,t}}} )} ]}_ + }} \|} ]} } } )^2} &\le {\vartheta _3}T + \frac{{2nR{{( \mathbb{X} )}^2}{\vartheta _4}}}{{{\gamma _0}}}T + \frac{{4n{G_1}R( \mathbb{X} ){\vartheta _4}{\alpha _0}\mu }}{{{\gamma _0}( {1 - \mu } )}}{T^{3/2}}  + \frac{{4nG_1^2{\vartheta _4}\alpha _0^2\mu }}{{{\gamma _0}( {1 - \mu } )}}T\log ( T )\\
&\;\; + \frac{{2n\tilde p\sqrt C R( \mathbb{X} ){\vartheta _4}{s_0}\mu }}{{{\gamma _0}( {1 - \mu } )}}T + \frac{{4n{{\tilde p}^2}CG_2^2{{\tilde \vartheta }_4}s_0^2{\mu ^2}}}{{1 - {\mu ^2}}}T.
\label{Theorem2-proof-ii-eq4}
\end{flalign}
From \eqref{Theorem2-proof-ii-eq4}, we know that \eqref{theorem2-eq3} holds. 

(iii)
Combining \eqref{Theorem1-proof-iii-eq2}, \eqref{Theorem1-proof-ii-eq4}, \eqref{Theorem2-proof-i-eq3}, and \eqref{Theorem2-proof-ii-eq1}--\eqref{Theorem2-proof-ii-eq2}, we have
\begin{flalign}
\nonumber
\sum\limits_{i = 1}^n {\sum\limits_{t = 1}^T {{\mathbf{E}_\mathcal{C}}[ {\| {{{[ {{g_{i,t}}( {{x_{i,t}}} )} ]}_ + }} \|} ]} } &\le \frac{{2nR{{( \mathbb{X} )}^2}}}{{{\varsigma _s}{\gamma _0}}} + \frac{{4n{G_1}R( \mathbb{X} ){\alpha _0}\mu }}{{{\varsigma _s}{\gamma _0}( {1 - \mu } )}}\sqrt T  + \frac{{4nG_1^2\alpha _0^2\mu }}{{{\varsigma _s}{\gamma _0}( {1 - \mu } )}}\log ( T ) \\
&\;\; + \frac{{2n\tilde p\sqrt C R( \mathbb{X} ){s_0}\mu }}{{{\varsigma _s}{\gamma _0}( {1 - \mu } )}}.
\label{Theorem2-proof-iii-eq1}
\end{flalign}
Combining \eqref{Lemma6-eq3}, \eqref{Theorem2-proof-i-eq2}, and \eqref{Theorem2-proof-i-eq3}, we have
\begin{flalign}
\nonumber
\mathbf{E}_\mathcal{C}[{{\rm{Net} \mbox{-} \rm{CCV}}( T )}] & \le n{G_2}{\vartheta _1} + \frac{{2{\vartheta _5}{\alpha _0}\mu }}{{1 - \mu }}\sqrt T  + \frac{{2nR{{( \mathbb{X} )}^2}{\vartheta _6}}}{{{\varsigma _s}{\gamma _0}}} + \frac{{4n{G_1}R( \mathbb{X} ){\vartheta _6}{\alpha _0}\mu }}{{{\varsigma _s}{\gamma _0}( {1 - \mu } )}}\sqrt T \\
&\;\; + \frac{{4nG_1^2{\vartheta _6}\alpha _0^2\mu }}{{{\varsigma _s}{\gamma _0}( {1 - \mu } )}}\log ( T ) + \frac{{2n\tilde p\sqrt C R( \mathbb{X} ){\vartheta _6}{s_0}\mu }}{{{\varsigma _s}{\gamma _0}( {1 - \mu } )}} + \frac{{n\tilde p\sqrt C {G_2}{{\tilde \vartheta }_2}{s_0}\mu }}{{1 - \mu }}.
\label{Theorem2-proof-iii-eq2}
\end{flalign}
From \eqref{Theorem2-proof-iii-eq2}, we know that \eqref{theorem2-eq4} holds. 

\bibliographystyle{IEEEtran}
\bibliography{reference_online}

\end{document}